\documentclass{birkau}

\usepackage{amsmath,amssymb,url}
\usepackage{amsfonts}
\usepackage{mathrsfs}
\usepackage{yfonts}
\usepackage{upgreek}
\usepackage{color}
\usepackage{tikz}
\usepackage{comment}

\numberwithin{equation}{section}

\theoremstyle{plain}
\newtheorem{theorem}{Theorem}[section]
\newtheorem{lemma}[theorem]{Lemma}
\newtheorem{proposition}[theorem]{Proposition}
\newtheorem{corollary}[theorem]{Corollary}

\theoremstyle{definition}
\newtheorem{definition}[theorem]{Definition}

\newtheorem{remark}[theorem]{Remark}
\newtheorem{example}[theorem]{Example}

\newcommand{\meet}{\wedge}
\newcommand{\join}{\vee}
\newcommand\upset{\mathord{\uparrow}}
\newcommand\downset{\mathord{\downarrow}}
\newcommand\rad[1]{\mathscr{R}({#1})}
\newcommand\bool[1]{\mathscr{B}({#1})}
\newcommand\mtl{\textsf{MTL}}
\newcommand\gmtl{\textsf{GMTL}}
\newcommand\sbp{\textsf{srDL}}
\newcommand\sbph{\textsf{srDL}_{\textsf{H}}}
\newcommand\ibp{\textsf{sIDL}}
\newcommand\psh{\textsf{GMTL}}
\newcommand\Xd[1]{\mathcal{S}({#1})}
\newcommand\Ad[1]{\mathcal{A}({#1})}
\newcommand{\Xdf}{\mathcal{S}}
\newcommand{\Adf}{\mathcal{A}}
\newcommand{\comp}{{\mathsf{c}}}
\newcommand\scrR{\mathscr{R}}
\newcommand\scrB{\mathscr{B}}
\newcommand\scrC{\mathscr{C}}
\newcommand\fr[1]{\mathfrak{#1}}
\newcommand\site[1]{\mathcal{S}_\fr{{#1}}}

\newcommand\Da[1]{\mathcal{F}_{\bf {#1}}}
\newcommand{\res}{\Rightarrow}
\newcommand{\resrad}{\Rightarrow^{\rad{\bf A}}}
\newcommand{\prodrad}{\bullet^{\rad{\bf A}}}
\newcommand{\upup}{\upupsilon}
\newcommand{\SnuX}{{\bf S} \otimes^{\Delta}_{\Upupsilon} {\bf X}}
\newcommand{\Qtau}{\mathcal{Q}^{\tau}}
\newcommand{\F}{\mathcal{F}}
\newcommand{\G}{\mathcal{G}}

\newcommand{\ovr}{\slash}
\newcommand{\under}{\backslash}

\begin{document}


\title[A topological approach to MTL-algebras]{A topological approach to MTL-algebras}

\author[Wesley Fussner]{Wesley Fussner}
\address{Department of Mathematics\\
University of Denver\\Denver, Colorado 80231\\USA}
\email{wesley.fussner@du.edu}

\author[Sara Ugolini]{Sara Ugolini}
\address{Department of Computer Science\\ University of Pisa\\Pisa, Italy}
\email{sara.ugolini@di.unipi.it}

\thanks{We would like to thank Nick Galatos for a number of helpful suggestions regarding this work.}




\begin{abstract}
We give a dualized construction of Aguzzoli-Flaminio-Ugolini of a large class of MTL-algebras from quadruples $({\bf B},{\bf A},\join_e,\delta)$, consisting of a Boolean algebra ${\bf B}$, a generalized MTL-algebra ${\bf A}$, and maps $\join_e$ and $\delta$ parameterizing the connection between these two constituent pieces. Our dualized construction gives a uniform way of building the extended Priestley spaces of MTL-algebras in this class from the Stone spaces of their Boolean skeletons, the extended Priestley spaces of their radicals, and a family of maps connecting the two. In order to make this dualized construction possible, we also present novel results regarding the extended Priestley duals of MTL-algebras and GMTL-algebras, in particular emphasizing their structure as Priestley spaces enriched by a partial binary operation.
\end{abstract}

\maketitle


\section{Introduction}
In lattice theory, triples constructions date back to Chen and Gr\"atzer's 1969 decomposition theorem for Stone algebras: each Stone algebra is characterized by the triple consisting of its lattice of complemented elements, its lattice of dense elements, and a map associating these structures \cite{ChenGratzer}. There is a long history of dual analogues of this construction, with Priestley providing a conceptually-similar treatment on duals in 1972 \cite{Pr3} and Pogel also exploring dual triples in his 1998 thesis \cite{PogelThesis}. At the same time, triples decompositions have been extended to account for richer algebraic structures. For example, Montagna-Ugolini \cite{MontUgolini} and Aguzzoli-Flaminio-Ugolini \cite{AguzFlamUgol} have provided similar triples decompositions for large classes of monoidal t-norm logic (MTL) algebras, the algebraic semantics for monoidal t-norm based logic \cite{EG01}, while \cite{BusanicheCignoliMarcos} and \cite{BusanMarcUgo} develop analogous triples representations for classes of residuated lattices. The aim of this paper is to provide a duality-theoretic perspective on these constructions, showing that Stone-Priestley duality offers a clarifying framework that sheds light on the representation of \cite{AguzFlamUgol}.

Putting aside the special case of Heyting algebras, Priestley duality has played a relatively minor role in the theory of distributive residuated lattices. This is perhaps due to the fact that existing Priestley-based dualities for distributive residuated lattices (see, e.g., \cite{Urquhart,GalatosThesis,Celani}) typically encode the monoid operation and its residuals by a ternary relation on duals, thereby attenuating the pictorial insights Priestley duality affords. However, in some settings the relation dualizing the monoid operation and its residuals can be understood as the graph of a (sometimes partially-defined) function \cite{FussPalm,Gehrke}, and this presents new approaches. The algebras in play in \cite{AguzFlamUgol} all possess functional Priestley duals in the aforementioned sense, and this facilitates our work dualizing the construction of the latter paper. Aligned with this approach, we develop a presentation of extended Priestley duals for $\mtl$-algebras that emphasized their character as partial algebras equipped with a topology. We also apply these ideas to obtain a similar presentation of the extended Priestley duals of $\gmtl$-algebras, where the aforementioned partial operations on duals are total.

Our main contribution consists of a construction of the extended Priestley duals of \sbp-algebras, a large class of ${\sf MTL}$-algebras including: G\"odel algebras, product algebras, the variety {\sf DLMV} generated by perfect MV-algebras as well as the variety generated by perfect {\sf MTL}-algebras (introduced as {\sf IBP$_{0}$} in \cite{NogueraThesis}, renamed {\textsf sIDL} in \cite{UgoliniThesis}), pseudocomplemented {\sf MTL}-algebras and the variety {\sf NM}$^{- }$ of nilpotent minimum algebras without negation fixpoint. Each \sbp-algebra enjoys a representation as a triple (augmented by an operator to form a quadruple, as shown in \cite{AguzFlamUgol}) consisting of its Boolean skeleton, its radical (i.e., the intersection of its maximal deductive filters), and an \emph{external join} connecting them (together with an additional nucleus on the radical). We construct the extended Priestley dual of each \sbp-algebra from the Stone space associated to its Boolean skeleton, the extended Priestley dual of its radical, and a collection of connecting maps, via a {\em rotation construction} that is similar to the one used in \cite{FussnerGalatos} (therein called the {\em reflection construction}, and used to provide a dualized version of the categorical equivalence present in the same paper between bounded Sugihara monoids and G\"odel algebras enriched with a Boolean constant).

The paper is organized as follows. At the outset, Section \ref{sec:prelim} offers necessary background material on residuated lattices, {\sf MTL}-algebras, and quadruple constructions needed in the sequel. Then Section \ref{sec:duality} develops material on extended Priestley duality as specialized to {\sf MTL}-algebras and {\sf GMTL}-algebras, in particular articulating this duality in terms of Priestley spaces enriched by a (possibly partial) binary operation. Insofar as the authors are aware, this perspective on the extended Priestley duality for {\sf MTL}-algebras and {\sf GTML}-algebras is new. Section \ref{sec:filterpairs} begins to outline technical material regarding the prime filters of $\sbp$-algebras, and Section \ref{sec:multiplyingfilters} focuses on how the multiplicative structure of $\sbp$-algebras may be captured in its extended Priestley space. Section \ref{sec:dualquadruples} introduces the appropriate dual analogues of the algebraic quadruples of \cite{AguzFlamUgol}, and constructs the extended Priestley duals of $\sbp$-algebras from them. 

\section{Preliminaries}\label{sec:prelim}

In order to summarize background material and fix notation, we first discuss preliminary material on residuated algebraic structures. As a general reference on residuated structures, we refer the reader to the standard monograph \cite{GJKO}.

\subsection{Residuated lattices and $\sbp$-algebras}

An algebra ${\bf A} = (A,\meet,\join,\cdot,\backslash,\slash,1)$ is called a \emph{residuated lattice} when the $(\meet,\join)$-reduct of ${\bf A}$ is a lattice, the $(\cdot,1)$-reduct of ${\bf A}$ is a monoid, and for all $a,b,c\in A$,
$$b\leq a\backslash c \iff a\cdot b\leq c\iff a\leq c\slash b.$$
The latter condition is typically called the \emph{law of residuation}.

We say that a residuated lattice ${\bf A}= (A,\meet,\join,\cdot,\backslash,\slash,1)$ is
\begin{itemize}
\item \emph{commutative} if $a\cdot b = b\cdot a$ for all $a,b\in A$,
\item \emph{integral} if $a\leq 1$ for all $a\in A$,
\item \emph{distributive} if $(A,\meet,\join)$ is a distributive a lattice, and
\item \emph{semilinear} if ${\bf A}$ is a subdirect product of residuated lattices whose lattice reducts are chains.
\end{itemize}

Note that when ${\bf A}$ is commutative, it satisfies the identity $a\backslash b = b\slash a$, and hence the operations $\backslash$ and $\slash$ coincide. In this case, we write $a\to b$ for the common value of $a\backslash b$ and $b\slash a$. If $a,b\in A$, we also sometimes write the product $a\cdot b$ as $ab$. The classes of residuated lattices, commutative residuated lattices, and commutative integral distributive residuated lattices are denoted $\textsf{RL}$, $\textsf{CRL}$, and $\textsf{CIDRL}$, respectively. We also sometimes refer to the members of these classes as RLs, CRLs, and CIDRLs.

When the lattice reduct of a residuated lattice ${\bf A}$ is bounded, we refer to the expansion of ${\bf A}$ by constants $\bot$ and $\top$ designating these bounds as a \emph{bounded} residuated lattice. Bounded residuated lattices are term-equivalent to expansions of RLs by only the least element $\bot$ due to the fact that every bounded RL satisfies $\top=\bot\slash\bot$. Because $1$ designates the top element of a CIDRL, the constant designating the bottom element in a bounded CIDRL will be denoted by $0$. For bounded CIDRLs, we define a unary negation operation $\neg$ by $\neg a = a\to 0$. We may also define a binary operation $\oplus$ by $a\oplus b = \neg (\neg a \cdot \neg b)$. The operation $\oplus$ is associative and commutative.

We will call semilinear CIDRLs \emph{generalized monoidal t-norm based logic algebras}, or $\gmtl$-algebras for short. These algebras are sometimes called \emph{prelinear semihoops} in the literature, and are of fundamental importance to this study. The class of $\gmtl$-algebras will be denoted by $\gmtl$. Bounded $\gmtl$-algebras are called \emph{monoidal t-norm based logic algebras} or, more briefly, \emph{\mtl-algebras}. The class of $\mtl$-algebras will be denoted $\mtl$, and gives the equivalent algebraic semantics \cite{BlokPigozzi} of Esteva and Godo's monoidal t-norm based logic (as introduced in \cite{EG01}). An \mtl-algebra is called an \emph{$\sbp$-algebra} if it additionally satisfies the identities
$$\neg (a^2) \to (\neg\neg a\to a) = 1\text{ and } (2a)^2 = 2(a^2),$$
where $a^2$ and $2a$ abbreviate $a\cdot a$ and $a\oplus a$, respectively. $\sbp$-algebras satisfying the involutivity identity $\neg\neg a = a$ are called \emph{$\ibp$-algebras}. The classes of $\sbp$-algebras and $\ibp$-algebras are denoted $\sbp$ and $\ibp$.

The law of residuation is not \emph{a priori} an equational condition, and yet $\textsf{RL}$ forms a variety (and consequently so do \textsf{CRL} and \textsf{CIDRL}). $\gmtl$-algebras turn out to be exactly those CIDRLs satisfying the identity $(a\to b)\join (b\to a)=1$, and therefore $\gmtl$ likewise forms a variety. It follows from this that the classes $\mtl$, $\sbp$, and $\ibp$ form varieties too. Each of these varieties may moreover be considered as a category in which the objects are the algebras belonging to the appropriate class, and the morphisms are the algebraic homomorphisms in the appropriate signature. In the sequel, we freely regard the varieties that we consider as categories in this manner, and we make no distinction between a variety and the category so obtained. In particular, we use \gmtl, \mtl, $\sbp$, and $\ibp$ to refer to the categories just described as well as the varieties.

The following proposition gives some useful properties that hold in the varieties discussed above. 
\begin{proposition}
Let ${\bf A}$ be an algebra in $\gmtl$, $\mtl$, $\sbp$, or $\ibp$. Then the following hold for all $a,b,c\in A$.
\begin{enumerate}
\item $a\cdot (a\to b)\leq b$.
\item If $a\leq b$, then $a\cdot c\leq a\cdot b$, $c\to a\leq c\to b$, and $b\to c\leq a\to c$.
\item $a\cdot (b\join c) = (a\cdot b)\join (a\cdot c)$.
\item $a\cdot (b\meet c) = (a\cdot b)\meet (a\cdot c)$.
\item $a\to (b\join c) = (a\to b)\join (a\to c)$.
\item $a\to (b\meet c) = (a\to b)\meet (a\to c).$
\item $(a\join b)\to c = (a\to c)\meet (b\to c).$
\item $(a\meet b)\to c = (a\to c)\join (b\to c)$.
\item $(a\cdot b)\to c = a\to (b\to c)$.
\item $a\cdot b\leq a\meet b$.
\end{enumerate}
If ${\bf A}$ is in $\mtl$, $\sbp$, or $\ibp$, then we additionally have
\begin{enumerate}
\item[(11)] $\neg (a\meet b) = \neg a\join\neg b.$
\item[(12)] $\neg (a\join b) = \neg a \meet\neg b.$
\item[(13)] $a\meet\neg a\leq b\join\neg b.$
\item[(14)] $a\leq \neg\neg a.$
\end{enumerate}
\end{proposition}

We say that an $\mtl$-algebra ${\bf A}$ \emph{has no zero divisors} if for all $a,b\in A$, $a\cdot b = 0$ implies $a=0$ or $b=0$. An $\mtl$-algebra is called an $\textsf{SMTL}$-algebra if it satisfies the identity $a\meet\neg a = 0$. As specialized to chains, the following appears in \cite[Proposition 4.14]{NogueraThesis}.

\begin{proposition}\label{prop:zerodiv}
Let ${\bf A}$ be an \mtl-algebra. Then ${\bf A}$ has no zero divisors if and only if ${\bf A}$ is a directly-indecomposable $\textsf{SMTL}$-algebra.
\end{proposition}

\begin{proof}
Suppose first that ${\bf A}$ has no zero divisors, and let $a\in A$. Then $a\cdot \neg a = a\cdot (a\to 0)=0$, so by hypothesis $a=0$ or $\neg a = a\to 0 = 0$. It follows immediately that $a\meet\neg a = 0$, so ${\bf A}$ is an ${\sf SMTL}$-algebra. If ${\bf A}={\bf A}_1\times {\bf A}_2$, where ${\bf A}_1$ and ${\bf A}_2$ are nontrivial $\mtl$-algebras, then the elements $(1,0),(0,1)\in A$ would satisfy $(1,0)\cdot (0,1)=(0,0)$ despite the fact that $(1,0),(0,1)$ are nonzero elements. This contradicts ${\bf A}$ having no zero divisors, whence ${\bf A}$ is directly indecomposable.
For the converse, it is well known that directly indecomposable ${\sf SMTL}$-algebras are ordinal sums of the kind ${\bf 2} \oplus {\bf H}$, with ${\bf H}$ a ${\sf GMTL}$-algebra (see for instance \cite{AguzFlamUgol}), thus they have no zero divisors.
\end{proof}

Given a CIDRL ${\bf A} = (A,\meet,\join,\cdot,\to,1)$, a map $\delta\colon A\to A$ is said to be a \emph{nucleus} if it is a closure operator that satisfies $\delta (a)\cdot\delta(b)\leq \delta(a\cdot b)$.  If $\delta$ is a nucleus on ${\bf A}$, then the algebra ${\bf A}_\delta = (\delta[A],\meet,\join_\delta, \cdot_\delta, \to, \delta(1))$ (where $a\cdot_\delta b = \delta(a\cdot b)$ and $a\join_\delta b = \delta(a\join b)$) is also a CIDRL. The resulting CIDRL is called the \emph{nuclear image} of ${\bf A}$.

\subsection{Radicals, Boolean elements, and algebraic quadruples}\label{sec:quadruples}

Given an $\sbp$-algebra ${\bf A}$, the \emph{radical} of ${\bf A}$ is the intersection of the maximal deductive filters of ${\bf A}$. The radical of an $\sbp$-algebra ${\bf A}$ forms a $\gmtl$-algebra with the operations inherited from ${\bf A}$, and we denote this $\gmtl$-algebra by $\rad{\bf A}$. For any $\sbp$-algebra ${\bf A}$, the radical of ${\bf A}$ coincides with the set
\begin{equation}\label{eq:radical}\{x\in A : \neg x < x\},\end{equation}
see for example \cite{AguzFlamUgol}.
 We denote also by $\bool{\bf A}$ the \emph{Boolean skeleton} of ${\bf A}$, i.e., the largest Boolean algebra contained in ${\bf A}$, again with operations inherited from ${\bf A}$. We call the elements of $\bool{\bf A}$ \emph{Boolean elements}. The manner in which $\bool{\bf A}$ interacts with ${\bf A}$ is summarized below.

\begin{lemma}\cite[Lemma 1.5]{BusanicheCignoliMarcos}\label{lem:Boolean elements}
Let ${\bf A}$ be an $\sbp$-algebra. Then
\begin{enumerate}
\item If $u\in\bool{\bf A}$, then $\neg u\in\bool{\bf A}$ and $\neg\neg u = u$.
\item An element $u\in A$ is Boolean if and only if $u\join\neg u=1$.
\end{enumerate}
If $u\in\bool{\bf A}$ and $a,b\in A$, then
\begin{enumerate}
\item $u\cdot a = u\meet a$,
\item $u\to a = \neg u \join a$,
\item $a=(a\meet u)\join (a\meet\neg u)$,
\item If $a\meet b\geq \neg u$ and $u\meet a = u\meet b$, then $a=b$.
\end{enumerate}
\end{lemma}
The \emph{coradical} of a $\sbp$-algebra ${\bf A}$ is the set $\scrC({\bf A}) = \{ x \in {\bf A}: \neg x \in \rad{\bf A}\}$. The following properties hold.
\begin{lemma}[\cite{AguzFlamUgol},\cite{UgoliniThesis}]\label{lemma:corad}
Let $\scrC({\bf A})$ be the coradical of an \sbp-algebra ${\bf A}$. Then:
\begin{enumerate}
\item $\scrC({\bf A}) = \{\neg x : x \in \rad{\bf A}\} = \{x\in A : x < \neg x\}$.
\item For every $y \in \rad{\bf A}$, $x \in \scrC({\bf A})$, $x < y$.
\item If ${\bf A}$ is directly indecomposable, ${\bf A} \cong \rad{\bf A} \cup \scrC({\bf A})$.
\end{enumerate}
\end{lemma}

Every element $a$ of an $\sbp$-algebra ${\bf A}$ can be expressed by means of an element $u \in \bool{\bf A}$ and an element $x \in \rad{\bf A}$ as 
\begin{equation}\label{eq:el}a = (u \lor \neg x) \land (\neg u \lor x)\end{equation}
Note that a Boolean element $u$ and a radical element $x$ of an $\sbp$-algebra, we have that $(u \lor \neg x) \land (\neg u \lor x)=(u\meet x)\join (\neg u\meet\neg x)$. These representations of elements inspired the decomposition results developed in \cite{AguzFlamUgol}, which we recall presently.

\begin{definition}\label{def:algquad}A \emph{wdl-admissible map} on a $\gmtl$-algebra ${\bf A}$ is a nucleus $\delta\colon A\to A$ that preserve both $\meet$ and $\join$. A \emph{GMTL-quadruple} (hereinafter \emph{algebraic quadruple}) is an ordered quadruple $({\bf B}, {\bf A}, \lor_{e}, \delta)$, where ${\bf B}$ is a Boolean algebra, ${\bf A}$ is a $\gmtl$-algebra, $B\cap A = \{1\}$, $\delta$ is wdl-admissible on ${\bf A}$, and $\vee_e:B\times A\to A$ is an \emph{external join}, i.e. it satisfies the following conditions where, for fixed $u \in B$ and $x\in A$, $\upnu_u(y) = u\vee_e y$ and $\uplambda_x(v) = v\vee_e x$:
\begin{itemize}
\item[(V1)] For every $u \in B$, and $x\in A$, 
$\upnu_u$ is an endomorphism of ${\bf A}$ and the map $\uplambda_x$ is a lattice homomorphism from (the lattice reduct of) ${\bf B}$ into (the lattice reduct of) ${\bf A}$.
\item[(V2)] 
$\upnu_0$ is the identity on ${\bf A}$ and $\upnu_1$ is constantly equal to $1$. (Note that since $\upnu_u(x) = \uplambda_x(u)$, for all $x \in A$,  $\uplambda_x(1) = 1$ and $\uplambda_x(0) = x$).
\item[(V3)] For all $u,v \in B$ and for all $x,y \in A$, $\upnu_u(x) \vee \upnu_{v} (y) = \upnu_{u\vee v} (x \vee y) = \upnu_u(\upnu_{v} (x \vee y))$.
\end{itemize}
\end{definition}
Such algebraic triples constitute a category $\mathcal{Q}_{\psh}$, whose morphisms are given by  {\em good morphism pairs}, i.e., if $({\bf B}, {\bf A}, \vee_e, \delta)$ and $({\bf B}', {\bf A}', \vee_e', \delta')$ are two triplets, a pair $(f,g)$ is a good morphism pair if $f:{\bf B}\to{\bf B}'$ is a Boolean homomorphism, $g:{\bf A}\to{\bf A}'$ is a GMTL-homomorphism, for every $(u,x)\in B\times A$, $g(u\vee_e x)=f(u)\vee_e' g(x)$, and finally $g(\delta(x)) = \delta'(g(x))$.


For an algebraic quadruple $({\bf B}, {\bf A}, \lor_{e}, \delta)$, define an equivalence relation $\sim$ on on $B\times A$ by $(u,x)\sim (v,y)$ if and only if $u=v$, $\upnu_{\neg u}(x) = \upnu_{\neg u}(y)$, and $\upnu_{u}(\delta(x)) = \upnu_u(\delta(y))$. 

The central result of \cite{AguzFlamUgol} defines the algebra ${\bf B}\otimes_e^\delta{\bf A}=(B\times A/\mathord\sim, \odot, \Rightarrow, \sqcap, \sqcup, [0,1], [1,1])$ with operations given by: \vspace{.3cm} 

\noindent
$[u,x]\odot[v, y]=[u\land v,  \upnu_{u\vee\neg v}(y\to x)\wedge \upnu_{\neg u\vee v}(x\to y)\wedge \upnu_{\neg u\vee \neg v}(x\cdot y)]$
\vspace{.2cm}

\noindent
$[u,x]\Rightarrow[v, y]=[u\to v, \upnu_{u\vee v}(\delta(y)\to \delta(x))\wedge \upnu_{\neg u\vee v}(\delta(x\cdot y))\wedge \upnu_{\neg u\vee \neg v}(x\to y)]$
\vspace{.2cm}

\noindent
$[u,x]\sqcap[v, y]=[u\land v, \upnu_{u\vee v}(x\vee y)\wedge \upnu_{u\vee\neg v}(x)\wedge \upnu_{\neg u\vee v}(y)\wedge \upnu_{\neg u\vee \neg v}(x\land y)]$
\vspace{.2cm}

\noindent
$[u,x]\sqcup[v, y]=[u\vee v, \upnu_{u\vee v}(x\land y)\wedge \upnu_{u\vee\neg v}(y)\wedge \upnu_{\neg u\vee v}(x)\wedge \upnu_{\neg u\vee \neg v}(x\vee y)]$
\vspace{0.3cm}\\
\noindent
${\bf B}\otimes_e^\delta{\bf A}$ is an $\sbp$-algebra.
Moreover, given any subvariety ${\textsf H}$ of $\psh$, let $\sbph$ be the full subcategory of $\sbp$ consisting of algebras whose radical is in $\textsf{H}$, and $\mathcal{Q}_{\textsf H}$ the full subcategory of $\mathcal{Q}_{\psh}$ made of quadruples $({\bf B},{\bf A},  \lor_{e}, \delta)$ with ${\bf A} \in {\textsf H}$. $\mathcal{Q}_{\textsf H}$ and $\sbph$ are equivalent via the functor $\Phi_{\textsf A}: \sbph \to \mathcal{Q}_{\psh}$ given by
\begin{align*}
 \Phi_{\textsf H}({\bf A})&= (\scrB({\bf A}), \scrR({\bf A}), \lor, \neg\neg)\\ 
\Phi_{\textsf H}(k)& =(k_{\restriction_{\scrB({\bf A})}}, k_{\restriction_{\scrR({\bf A})}})
\end{align*}
with reverse functor $\Xi_{\textsf H}: \mathcal{Q}_{\psh} \to \sbph$ defined by
\begin{align*}
\Xi_{\textsf H}(({\bf B}, {\bf A}, \vee_e,\delta))& ={\bf B}\otimes_e^{\delta} {\bf A}\\
\Xi_{\textsf H}(f,g)([u, x])&=[f(u), g(x)].
\end{align*}
Notice that given any quadruple $( {\bf B}, {\bf A}, \lor_{e}, \delta)$, the external join $\lor_{e}$ can be described as the indexed family of maps $\{\upnu_{b} \}_{b \in B}$.

\section{Duality theory for $\mtl$, $\gmtl$, and $\sbp$}\label{sec:duality}

In this section, we recall Priestley duality for bounded distributive lattices and its extensions to account for various classes of bounded distributive residuated lattices. We further refine this extension to account for the omission of one of the lattice bounds, obtaining in particular a duality for $\gmtl$-algebras. We also provide a rendering of this duality in terms of functional dual spaces. 

\subsection{Priestley duality and its extensions}

An ordered topological space $(S,\leq,\tau)$ is said to be \emph{totally order-disconnected} if for each $x,y\in S$ with $x\not\leq y$ there exists a clopen up-set $U\subseteq S$ such that $x\in U$ and $y\notin U$. The compact, totally order-disconnected ordered topological spaces are called \emph{Priestley spaces}. As is well-known, the category \textsf{Pries} of Priestley spaces and continuous isotone maps is dually equivalent to the category \textsf{D$_{01}$} of bounded distributive lattices and bounded lattice homomorphisms \cite{Pr1,Pr2}. The functors $\Xdf$ and $\Adf$ demonstrating this dual equivalence are defined as follows.

For a bounded distributive lattice ${\bf D} = (D,\meet,\join,0,1)$, denote by $\Xd{D}$ the collection of prime filters of ${\bf D}$, and for each $a\in D$ set $\varphi_{\bf D}(a) = \{\fr{a}\in\Xd{D} : a\in \fr{a}\}$. Letting $\tau$ be the topology generated by $\{\varphi_{\bf D}(a),\varphi_{\bf D}(a)^\comp : a\in D\}$, we obtain that $\Xd{\bf D} = (\Xd{D},\subseteq,\tau)$ is a Priestley space. For a bounded distributive lattice homomorphism $f\colon {\bf D}_1\to {\bf D}_2$, we define a function $\Xd{f}\colon \Xd{{\bf D}_2}\to \Xd{{\bf D}_1}$ by $\Xd{f}(\fr{a}) = f^{-1}[\fr{a}]$. Then $\Xd{f}$ is a continuous isotone map.

For the reverse functor $\Adf$, given a Priestley space ${\bf S} = (S,\leq,\tau)$ let $\Ad{S}$ be the collection of clopen up-sets of ${\bf S}$. Then $\Ad{\bf S} = (\Ad{S},\cap,\cup,\emptyset, S)$ is a bounded distributive lattice. For a continuous isotone map $\alpha\colon {\bf S}_1\to {\bf S}_2$ between Priestley spaces, define $\Ad{\alpha}\colon\Ad{{\bf S}_2}\to\Ad{{\bf S}_1}$ by $\Ad{\alpha}(U)=\alpha^{-1}[U]$. The map $\Ad{\alpha}$ so defined is a bounded lattice homomorphism.

The set-up outlined above may be modified to accommodate the omission of one or both bounds from the signature on the lattice side of the duality, and considering only top-bounded distributive lattices will prove indispensable in this study. We denote by \textsf{D$_1$} the category of distributive lattices with a designated top element (possibly missing a bottom element) and lattice homomorphisms preserving the top element. A \emph{pointed Priestley space} is a structure $(S\leq,\top,\tau)$, where $(S,\leq,\tau)$ is a Priestley space and $\top$ is the greatest element with respect to $\leq$, and we denote the category of pointed Priestley spaces and continuous isotone maps preserving the top element by \textsf{pPries}. It turns out that \textsf{D$_1$} is dually equivalent to \textsf{pPries}. The functors demonstrating this are variants of $\Xdf$ and $\Adf$. For an object ${\bf D}$ of \textsf{D$_1$}, call $\fr{a}$ a \emph{generalized prime filter} of ${\bf D}$ if $\fr{a}$ is a prime filter of ${\bf D}$ or $\fr{a}=D$, and let $\Xd{D}$ be the collection of generalized prime filters of ${\bf D}$. For an object ${\bf S}$ of \textsf{pPries}, let $\Ad{S}$ be the collection of \emph{nonempty} clopen up-sets of ${\bf S}$. The functors resulting from these modifications yield the desired dual equivalence, and without danger of confusion we use $\Xdf$ and $\Adf$ for both the bounded and half-bounded cases.

The chief tool in the present investigation is the extension of Priestley duality to account for the expansion of the distributive lattice signature by a residuated pair $(\cdot,\to)$. Our treatment is essentially drawn from \cite{GalatosThesis} and \cite{Urquhart} (but see also \cite{Celani,CabCel,Goldblatt}). In order to describe the appropriate dual category, we introduce some notation. Let $({\bf S},R)$ be a structure where ${\bf S}$ is a Priestley space and $R$ is a ternary relation on $S$. For $U,V\subseteq S$, we define
$$R[U,V,-] =\{z\in S : (\exists x\in U)(\exists y\in V)(R(x,y,z))\}.$$
Also, for $z\in S$ set $R[z,V,-]:=R[\{z\},V,-]$ and $R[U,z,-]:=R[U,\{z\},-]$.
\begin{definition}\label{def:unpointed residuated space}
We call a structure $({\bf S},R,E)$ an \emph{unpointed residuated space} if ${\bf S}$ is a Priestley space, $R$ is a ternary relation on $S$, $E$ is a subset of $S$, and the following conditions hold for all $x,y,z,w,x',y',z'\in S$ and $U,V\in\Ad{S}$.
\begin{enumerate}
\item  $R(x,y,u)$ and $R(u,z,w)$ for some $u\in S$ if and only if $R(y,z,v)$ and $R(x,v,w)$ for some $v\in S$.
\item If $x'\leq x$, $y'\leq y$, and $z\leq z'$, then $R(x,y,z)$ implies $R(x',y',z')$.
\item If $R(x,y,z)$ does not hold, then there exist $U,V\in\Ad{S}$ such that $x\in U$, $y\in V$, and $z\notin R[U,V,-]$.
\item For all $U,V\in\Ad{S}$, each of $R[U,V,-]$, $\{z\in S : R[z,V,-]\subseteq U\}$, and $\{z\in S : R[B,z,-]\subseteq U\}$ are clopen.
\item $E\in\Ad{S}$ and for all $U\in\Ad{S}$ we have $R[U,E,-]=R[E,U,-]=U$.
\end{enumerate}
If ${\bf S}_1=(S_1,\leq_1,\tau_1,R_1,E_1)$ and ${\bf S}_2=(S_2,\leq_2,\tau_2,R_2,E_2)$ are unpointed residuated spaces, a map $\alpha\colon S_1\to S_2$ is a \emph{bounded morphism} provided it satisfies the following.
\begin{enumerate}
\item $\alpha$ is a continuous isotone map.
\item If $R_1(x,y,z)$, then $R_2(\alpha(x),\alpha(y),\alpha(z))$.
\item If $R_2(u,v,\alpha(z))$, then there exist $x,y\in S_1$ such that $u\leq\alpha(x)$, $v\leq\alpha(y)$, and $R_1(x,y,z)$.
\item For all $U,V\in\Ad{S_2}$ and all $x\in S_1$, if $R_1[x,\alpha^{-1}[U],-]\subseteq\alpha^{-1}[V]$, then $R_2[\alpha(x),U,-]\subseteq V$.
\item $\alpha^{-1}[E_2]\subseteq E_1$.
\end{enumerate}
We denote the category of unpointed residuated spaces and bounded morphisms by $\textsf{uRS}$.
\end{definition}
The proof of the following theorem may be found in \cite[Theorem 6.13]{GalatosThesis}.
\begin{theorem}\label{thm:duality for RL}
The category of bounded distributive residuated lattices with residuated lattice homomorphisms preserving the lattice bounds is dually equivalent to $\textsf{uRS}$.
\end{theorem}

The functors of the dual equivalence described by Theorem \ref{thm:duality for RL} are variants of $\Adf$ and $\Xdf$, enriched as follows. Given a residuated lattice ${\bf A}$, define the \emph{complex product} of filters $\fr{a}$, $\fr{b}$ of ${\bf A}$ by $\fr{a}\cdot\fr{b}=\{ab : a\in\fr{a},b\in\fr{b}\}$. Note that $\fr{a}\cdot\fr{b}\subseteq\fr{c}$ holds if and only if $\fr{a}\bullet\fr{b}\subseteq\fr{c}$, where
$$\fr{a}\bullet\fr{b}=\upset (\fr{a}\cdot\fr{b})=\{c\in A : \exists (a,b)\in\fr{a}\times\fr{b}, ab\leq c\}.$$
For a bounded residuated lattice ${\bf A}$ with bounded lattice reduct ${\bf D}$, we define a ternary relation $R$ on $\Xd{D}$ by $$R(\fr{a},\fr{b},\fr{c}) \mbox{ iff } \fr{a}\bullet\fr{b}\subseteq\fr{c}$$ and set $E=\{\fr{a}\in\Xd{D} : 1\in \fr{a}\}$. Then we define $\Xd{\bf A} = (\Xd{\bf D},R,E)$. For the enrichment of $\Adf$, given an unpointed residuated space ${\bf S} = (S,\leq,\tau,R,E)$, we define $\Ad{\bf S} = (\Ad{S,\leq,\tau},\cdot,\to,E)$, where for $U,V\in\Ad{S,\leq,\tau}$ we define 
\begin{equation}\label{eq:clprod}U\cdot V=R[U,V,-]\end{equation} 
\begin{equation}\label{eq:clres1}U\ovr V=\{x\in S : R[x,V,-]\subseteq U\}\end{equation}
\begin{equation}\label{eq:clres2}U\under V=\{x\in S : R[U,x,-]\subseteq V\}\end{equation}
For commutative residuated lattices, the latter two sets coincide and we may unpack the above definition to get
\begin{equation}\label{eq:clres}U\to V =\{x\in S : (\forall y,z)(y\in U\text{ and }R(x,y,z)\implies z\in V)\}\end{equation}
The duality articulated in Theorem \ref{thm:duality for RL} may be specialized to obtain dualities for a host of subcategories of bounded residuated lattices. The following correspondences are of particular interest to this study (see, e.g., \cite{CabCel}).

\begin{proposition}\label{prop:correspond}
Let ${\bf A}=(A,\meet,\join,\cdot,\backslash,\slash,1,\bot,\top)$ be a bounded distributive residuated lattice and ${\bf S}=(S,\leq,\tau,R,E)$ its dual space. In each of the following pairs of statements, (a) holds iff (b) holds.
\begin{enumerate}
\item
\begin{enumerate}
\item ${\bf A}$ is commutative.
\item For all $x,y,z\in S$, $R(x,y,z)$ iff $R(y,x,z)$.
\end{enumerate}
\item
\begin{enumerate}
\item ${\bf A}$ is integral.
\item $E=S$.
\end{enumerate}
\item In the presence of integrality and commutativity, 
\begin{enumerate}
\item ${\bf A}$ satisfies $1\leq (a\to b)\join (b\to a)$.
\item For all $x,y,z,v,w\in S$, if $R(x,y,z)$ and $R(x,v,w)$, then $y\leq w$ or $v\leq z$.
\end{enumerate}
\end{enumerate}
\end{proposition}

Due to the fact that semilinearity is axiomatized (modulo commutativity and distributivity) by $1\leq (a\to b)\join (b\to a)$, it follows from Proposition \ref{prop:correspond} that the unpointed residuated spaces corresponding to $\mtl$-algebras are exactly those satisfying the three conditions given above. We denote by $\mtl^\tau$ the full subcategory of ${\sf uRS}$ whose objects satisfy these three conditions. From the preceding remarks, the following is immediate.
\begin{theorem} $\mtl^\tau$ is dually equivalent to $\mtl$.
\end{theorem}
Although the theory outlined above is given in terms of bounded algebras, it may be modified to account for the half-bounded case. Our exposition of this modification is inspired by the extension in \cite{JR} of the Esakia duality for Heyting algebras to a duality for Brouwerian algebras. Given a $\gmtl$-algebra ${\bf A}$, we define a new algebra ${\bf A}_0:={\bf 2} \oplus {\bf A}$, whose universe is $A\cup\{0\}$. As in Proposition \ref{prop:zerodiv}, the $\mtl$-algebra ${\bf A}_0$ has no \emph{zero-divisors} in the sense that $a\cdot b=0$ implies $a=0$ or $b=0$. 

Conversely, if ${\bf A}=(A,\meet,\join,\cdot,\to,1,0)$ is an $\mtl$-algebra with no zero-divisors, then $A\setminus \{0\}$ is the universe of a $(\meet,\join,\cdot,\to,1)$-subalgebra of ${\bf A}$. The fact that $A\setminus \{0\}$ is closed under $\cdot$ follows immediately from that fact that ${\bf A}$ has no zero-divisors. On the other hand, the identity $b\leq a\to b$, which holds in all integral CRLs, guarantees that $b\leq a\to b\neq 0$ provided that $b\neq 0$, whence $A\setminus\{0\}$ is closed under $\to$. We denote the resulting $(\meet,\join,\cdot,\to,1)$-subalgebra by ${\bf A}^0$, and observe that it is a $\gmtl$-algebra. 
%
%
%

It is easy to see that $({\bf A}_0)^0\cong {\bf A}$ for any $\gmtl$-algebra ${\bf A}$, and likewise that $({\bf A}^0)_0\cong {\bf A}$ for any $\mtl$-algebra ${\bf A}$ without zero divisors. This correspondence may be lifted to morphisms in the following way. Given ${\bf A}$ and ${\bf B}$ $\gmtl$-algebras and $f\colon {\bf A}\to {\bf B}$ a homomorphism, $f$ extends uniquely to a morphism of $\mtl$-algebras $f_0\colon {\bf A}_0\to {\bf B}_0$ by setting $f_0(a)=a$ for $a\in A$ and $f_0(0_{\bf A})=0_{\bf B}$. Moreover, suppose that ${\bf A}$ and ${\bf B}$ are $\mtl$-algebras without zero divisors and $f\colon{\bf A}\to {\bf B}$ is a morphism of $\mtl$. If $a\in A$ with $a\neq 0$, then from $a\cdot (a\to 0)=0_{\bf A}$ we must have that $a\to 0 = 0$, and hence $f(a\to 0) = 0$. Observe that if $f(a)=0$, then by residuation we have $1\leq f(a)\to 0=f(a)\to f(0)=f(a\to 0)$, contradicting $f(a\to 0)=0$. It follows that $f[A\setminus\{0\}]\subseteq B\setminus\{0\}$, and hence that $f$ restricts to a morphism $f^0\colon {\bf A}^0\to {\bf B}^0$ of $\gmtl$-algebras. It is clear that $(f^{0})_{0} = f$, and moreover, if $f \neq f'$, they necessarily differ on some element in $A\setminus\{0\}$, thus $f^{0} \neq f^{'0}$. The upshot of these observations, recalling also Proposition \ref{prop:zerodiv}, is the following.

\begin{theorem}\label{lem:equivalence of mtl and gtml}
Let $\mtl_{div}$ be the full subcategory of $\mtl$-algebras without zero divisors and ${\sf SMTL}_{ind}$ be the full subcategory of directly indecomposable ${\sf SMTL}$-algebras. Then the categories $\mtl_{div}$, $\gmtl$, and ${\sf SMTL}_{ind}$ are equivalent.
\end{theorem}

The foregoing lemma provides a duality for $\gmtl$ via the previously discussed duality for $\mtl$, using the full subcategory $\mtl_{div}^{\tau}$ as a bridge. In more detail, each $\gmtl$-algebra ${\bf A}$ may be associated to the $\mtl$-algebra ${\bf A}_0$, which by Theorem \ref{thm:duality for RL} has a dual $\Xd{{\bf A}_0} \in \mtl_{div}^{\tau}$. Owing to the fact that $A$ is a prime filter of ${\bf A}_0$, $\Xd{{\bf A}_0}$ has a maximum element for any $\gmtl$-algebra ${\bf A}$. Additionally, if $f\colon{\bf A}\to{\bf B}$ is a morphism of $\gmtl$, then by construction $f_0^{-1}[B]=A$, whence $\Xd{f_0}\colon\Xd{{\bf B}_0}\to \Xd{{\bf A}_0}$ preserves the greatest element of $\Xd{{\bf B}_0}$.

On the other hand, if ${\bf S}=(S,\leq,\tau,R,E)$ is an object of $\mtl^\tau$ with a top element $\top$, then it is easy to see that the collection of nonempty clopen up-sets of ${\bf S}$ is closed under $\cdot$, $\to$, $\cap$, $\cup$, and contains $E$. It follows that the nonempty clopen up-sets form a $(\meet,\join,\cdot,\to,1)$-subalgebra of $\Ad{\bf S}$, and therefore form a $\gmtl$-algebra. If $\alpha\colon{\bf S}_1\to {\bf S}_2$ is a morphism between top-bounded objects of $\mtl^\tau$ that preserves the top element, then we also have that $\top_2\in\Ad{\alpha}({\sf U})=\alpha^{-1}[{\sf U}]$ for any nonempty clopen up-set ${\sf U}$ of ${\bf S}_2$, so $\Ad{\alpha}$ restricts to a morphism between the corresponding $\gmtl$-algebras of nonempty clopen up-sets. Thus $\mtl^\tau_{div}$ is the full subcategory of $\mtl^\tau$ of spaces with a top element and top-preserving morphism. This discussion leads us to the following definition.

\begin{definition}
We define a category $\gmtl^\tau$ as follows. The objects of $\gmtl^\tau$ are structures of the form $({\bf S},R,E,\top)$, where $({\bf S},R,E)$ is an object of $\mtl^\tau$ with a greatest element $\top$. \\
The morphisms of $\gmtl^\tau$ are maps $\alpha\colon ({\bf S}_1,R_1,E_1,\top_1)\to ({\bf S}_2,R_2,E_2,\top_2)$ between objects of $\gmtl^\tau$, where $\alpha$ is a bounded morphism considered as a map between unpointed residuated spaces, and $\alpha(\top_1)=\top_2$.
\end{definition}

From the preceding discussion, we obtain the following duality for $\gmtl$-algebras.

\begin{theorem}
$\gmtl$ and $\gmtl^\tau$ are dually equivalent categories.
\end{theorem}

\subsection{Functional residuated spaces for $\mtl$, $\gmtl$, and $\sbp$}\label{sec:functional}
Dualities for classes of distributive residuated lattices typically employ a ternary relation on dual structures as in the discussion above, in \cite{GalatosThesis, Celani, Urquhart}, and elsewhere. However, for semilinear residuated structures, the ternary relation on duals may be presented as a (possibly partially-defined) binary operation on the underlying Priestley dual of the lattice reduct. This \emph{functionality} of duals is thoroughly treated in \cite{Gehrke, FussPalm}, but as far as the authors are aware has not previously been explored in the context of $\mtl$ or $\gmtl$. The presentation of the ternary dual relation as a partial operation proves convenient in the sequel, so we now discuss functionality in the environment of $\mtl^\tau$ and $\gmtl^\tau$. The following technical lemma is fundamental, and its proof may be found, e.g., in \cite[Lemmas 6.8 and 6.9]{GalatosThesis}.
\begin{lemma}\label{lem:filter to prime filter}
Let ${\bf A}$ be a residuated lattice and let $\fr{a}$, $\fr{b}$, and $\fr{c}$ be filters of ${\bf A}$. Then the following hold.
\begin{enumerate}
\item The up-set of the complex product $\fr{a}\bullet\fr{b}=\upset\{a\cdot b : a\in\fr{a},b\in\fr{b}\}$ is also a filter.
\item If ${\bf A}$ has a distributive lattice reduct, $\fr{c}$ is prime, and $\fr{a}\bullet\fr{b}\subseteq\fr{c}$, then there exist prime filters $\fr{a}'$ and $\fr{b}'$ of ${\bf A}$ such that $\fr{a}\subseteq\fr{a}'$, $\fr{b}\subseteq\fr{b}'$, $\fr{a}'\bullet\fr{b}\subseteq\fr{c}$, and $\fr{a}\bullet\fr{b}'\subseteq\fr{c}$.
\end{enumerate}
\end{lemma}
The previous lemma provides that $\bullet$ is an operation on the collection of filters of ${\bf A}$ for any residuated lattice ${\bf A}$, but in general $\fr{a}\bullet\fr{b}$ may fail to be prime even when $\fr{a},\fr{b}\in\Xd{\bf A}$. However, algebras in $\mtl$ and $\gmtl$ satisfy the identity
$$a\to (b\join c) = (a\to b)\join (a\to c),$$
as a consequence of their semilinearity. In this environment, we have the following.
\begin{lemma}\label{lem:filtmultwelldef}
Let ${\bf A}$ be a $\gmtl$-algebra or an $\mtl$-algebra, and let $\fr{a},\fr{b}$ be nonempty filters of ${\bf A}$. Then if either one of $\fr{a}$ or $\fr{b}$ is prime, we have that either  $\fr{a}\bullet\fr{b}$ is prime or $\fr{a}\bullet\fr{b}=A$.
\end{lemma}

\begin{proof}
That $\fr{a}\bullet\fr{b}$ is a filter follows from Lemma \ref{lem:filter to prime filter}(1). For primality, suppose without loss of generality that $\fr{a}$ is prime and let $a\join b \in\fr{a}\bullet\fr{b}$. Then there exists $c\in\fr{a}$, $d\in\fr{b}$ with $c\cdot d\leq a\join b$. By residuation, $c\leq d\to (a\join b) = (d\to a)\join (d\to b)$, and since $\fr{a}$ is upward-closed and prime it follows that $d\to a\in\fr{a}$ or $d\to b\in\fr{a}$. Then $(d\to a)\cdot d = d\cdot (d\to a)\leq a\in\fr{a}\bullet\fr{b}$ or $(d\to b)\cdot d = d\cdot (d\to b)\leq b\in\fr{a}\bullet\fr{b}$ as desired. Hence either $\fr{a}\bullet\fr{b}=A$, or $\fr{a}\bullet\fr{b}$ is a prime filter of ${\bf A}$.
\end{proof}

\begin{corollary}\label{cor:partialop}
If ${\bf A}$ is an $\gmtl$-algebra, then $\bullet$ is a binary operation on $\Xd{A}$. If ${\bf A}$ is an $\mtl$-algebra, then $\bullet$ is gives a partial operation on $\Xd{A}$ and is undefined for $\fr{a},\fr{b}\in\Xd{A}$ only if $\fr{a}\bullet\fr{b}=A$.
\end{corollary}

\begin{proof}
Lemma \ref{lem:filtmultwelldef} provides that if $\fr{a}$ and $\fr{b}$ are prime filters of a $\gmtl$-algebra ${\bf A}$, then $\fr{a}\bullet\fr{b}$ is either a prime filter or $A$. Hence $\fr{a}\bullet\fr{b}\in\Xd{A}$ when ${\bf A}$ is a $\gmtl$-algebra, and $\fr{a}\bullet\fr{b}\in\Xd{A}$ provided that $\fr{a}\bullet\fr{b}\neq A$ when ${\bf A}$ is an $\mtl$-algebra.
\end{proof}
The following provides a mechanism for defining $\bullet$ on abstract duals of $\mtl$ and $\gmtl$-algebras.

\begin{lemma}\label{lem:least R successor}
Let ${\bf S} = (S,\leq,\tau,R,E)$ be an object of $\mtl^\tau$. If $\sf{x},\sf{y},\sf{z}\in S$ with $R(\sf{x},\sf{y},\sf{z})$, then there exists a least element $\sf{z}'\in S$ such that $R(\sf{x},\sf{y},\sf{z}')$. If ${\bf S}$ is in $\gmtl^\tau$, then for any $\sf{x},\sf{y}\in S$ there exists a least $\sf{z}'\in S$ with $R(\sf{x},\sf{y},\sf{z}')$.
\end{lemma}

\begin{proof}
By the duality for $\mtl$, there exists an $\mtl$-algebra ${\bf A}$ such that ${\bf S}\cong\Xd{\bf A}$. Let $\alpha\colon{\bf S}\to\Xd{\bf A}$ be the map witnessing this isomorphism. Then $\alpha(\sf{x})$, $\alpha(\sf{y})$, and $\alpha(\sf{z})$ are prime filters of ${\bf A}$ and $R^{\Xd{\bf A}}(\alpha(\sf{x}),\alpha(\sf{y}),\alpha(\sf{z}))$, i.e., $\alpha(\sf{x})\bullet\alpha(\sf{y})\subseteq\alpha(\sf{z})$. By Lemma \ref{lem:filtmultwelldef}, $\alpha({\sf x})\bullet\alpha({\sf y})$ is either a prime filter of ${\bf A}$ or $\alpha({\sf x})\bullet\alpha({\sf y})=A$. Because $\alpha({\sf x})\bullet\alpha({\sf y})\subseteq \alpha ({\sf z})\neq A$, it follows that $\alpha(\sf{x})\bullet\alpha({\sf y})\in\Xd{\bf A}$. This implies that $R^{\Xd{\bf A}}(\alpha({\sf x}),\alpha({\sf y}),\alpha({\sf x})\bullet\alpha({\sf y}))$. Because $\alpha^{-1}$ is an isomorphism with respect to $R$ as well, we also have that $R({\sf x},{\sf y},\alpha^{-1}(\alpha({\sf x})\bullet\alpha({\sf y})))$. Moreover, if ${\sf z}\in S$ with $R({\sf x},{\sf y},{\sf z})$, then by the isomorphism we have $\alpha({\sf x})\bullet\alpha({\sf y})\subseteq\alpha({\sf z})$, and because $\alpha$ is an order isomorphism we must have $\alpha^{-1}(\alpha({\sf x})\bullet\alpha({\sf y}))\subseteq\alpha^{-1}(\alpha({\sf z}))={\sf z}$. Thus ${\sf z}'=\alpha^{-1}(\alpha({\sf x})\bullet\alpha({\sf y}))$ is the least ${\sf z}\in S$ such that $R({\sf x},{\sf y},{\sf z})$ as desired. The result for ${\bf S}$ being an object of $\gmtl^\tau$ follows by the same argument.
\end{proof}

Owing to Lemma \ref{lem:least R successor}, we may define a partial operation $\bullet$ for ${\bf S}$ an object of $\gmtl^\tau$ or $\mtl^\tau$ by setting
\[
{\sf x}\bullet{\sf y} = \begin{cases}
\min\{{\sf z}\in S : R({\sf x},{\sf y},{\sf z})\}, & \text{if }\{{\sf z}\in S : R({\sf x},{\sf y},{\sf z})\}\neq\emptyset\\
\text{undefined}, & \text{otherwise}
\end{cases}
\]
Observe that if ${\bf S}$ is an object of $\gmtl^\tau$, then $\bullet$ is a total operation. Hereinafter we adopt the convention that whenever we write a statement involving $\bullet$, we assert the statement only for those instances when $\bullet$ is defined. Hence when we say that an identity or inequality holds in an object ${\bf S}$ of $\mtl^\tau$, we mean in particular that it holds in those instances for which all occurrences of $\bullet$ in the identity or inequality are defined.

\begin{lemma}\label{lem:bullet properties}
Let ${\bf S}$ be an object of $\mtl^\tau$ or $\gmtl^\tau$, and let ${\sf x},{\sf y},{\sf z}\in S$.
\begin{enumerate}
\item $R({\sf x},{\sf y},{\sf z})$ iff ${\sf x}\bullet{\sf y}\leq {\sf z}$.
\item Each of the following holds.
\begin{enumerate}
\item ${\sf x}\bullet ({\sf y}\bullet{\sf z}) = ({\sf x}\bullet {\sf y})\bullet{\sf z}$.
\item ${\sf x}\bullet {\sf y}={\sf y}\bullet {\sf x}$.
\item ${\sf x}\leq {\sf y}$ implies that ${\sf x}\bullet {\sf z}\leq {\sf y}\bullet {\sf z}$ and ${\sf z}\bullet {\sf x}\leq {\sf z}\bullet {\sf y}$.
\end{enumerate}
\end{enumerate}
\end{lemma}

\begin{proof}
For (1), note that if $R({\sf x},{\sf y},{\sf z})$, then by Lemma \ref{lem:least R successor} there is a least ${\sf z}'\in S$ with $R({\sf x},{\sf y},{\sf z}')$, and by definition $z'={\sf x}\bullet {\sf y}$. This immediately yields ${\sf x}\bullet {\sf y}\leq {\sf z}$. For the converse, if ${\sf x}\bullet {\sf y}$ is defined, then by definition $R({\sf x},{\sf y},{\sf x}\bullet {\sf y})$. If ${\sf x}\bullet {\sf y}\leq {\sf z}$, then by the isotonicity of $R$ in the third coordinate we have $R({\sf x},{\sf y},{\sf z})$ as well.

For (2), observe that by the dualities for $\mtl$ (respectively, $\gmtl$) there exists an $\mtl$-algebra (respectively, $\gmtl$-algebra) ${\bf A}$ such that ${\bf S}\cong \Xd{\bf A}$. Let $\alpha\colon{\bf S}\to\Xd{\bf A}$ be an isomorphism. Then the proof of Lemma \ref{lem:least R successor} shows that ${\sf x}\bullet {\sf y}=\alpha^{-1}(\alpha({\sf x})\bullet\alpha({\sf y}))$, whence $\alpha({\sf x}\bullet{\sf y})=\alpha({\sf x})\bullet\alpha({\sf y})$. The properties (a), (b), and (c) then follow from the fact that multiplication of filters of an $\mtl$-algebra or $\gmtl$-algebra are associative, commutative, and order-preserving.
\end{proof}

The fact that the multiplication of an $\mtl$- or $\gmtl$-algebra is residuated is reflected in the the order-preservation property of $\bullet$ as codified in Lemma \ref{lem:bullet properties}(2)(c), and it turns out that $\bullet$ possesses a partial residual as well.

\begin{proposition}
Let ${\bf A}$ be an $\mtl$-algebra or $\gmtl$-algebra, and suppose that $\fr{b},\fr{c}\in\Xd{\bf A}$ are such that there exists $\fr{a}\in\Xd{\bf A}$ with $\fr{a}\bullet\fr{b}\subseteq\fr{c}$. Then there is a greatest such $\fr{a}$, and it is given by
$$\fr{b}\res \fr{c} := \bigcup \{\fr{a}\in\Xd{A} : \fr{a}\bullet \fr{b}\subseteq \fr{c}\}.$$
Moreover, $\fr{a}\bullet\fr{b}\subseteq \fr{c}$ if and only if $\fr{a}\subseteq\fr{b}\res\fr{c}$.
\end{proposition}

\begin{proof}
We show first that $\fr{b}\res \fr{c}$ is a prime filter. For upward-closure, let $a\in \fr{b}\res \fr{c}$ and $a\leq b$. Then there exists $\fr{a}\in\Xd{A}$ with $a\in \fr{a}$ and $\fr{a}\bullet\fr{b}\subseteq \fr{c}$. But $\fr{a}$ is upward-closed, so $b\in \fr{a}$, and this gives that $b\in \fr{b}\res \fr{c}$.

For closure under meets, let $a,b\in \fr{b}\res \fr{c}$. Then there exist $\fr{a}_1,\fr{a}_2\in\Xd{A}$ with $a\in \fr{a}_1$, $b\in \fr{a}_2$, $\fr{a}_1\bullet \fr{b}\subseteq \fr{c}$, and $\fr{a}_2\bullet \fr{b}\subseteq \fr{c}$. Denote by $\fr{a}_1\join \fr{a}_2$ the filter generated by $\fr{a}_1\cup \fr{a}_2$. Then $a\meet b\in \fr{a}_1\join \fr{a}_2$. We claim that $(\fr{a}_1\join \fr{a}_2)\bullet \fr{b}\subseteq \fr{c}$. To see this, let $e\in (\fr{a}_1\join \fr{a}_2)\bullet\fr{b}$. Then there exist $c\in \fr{a}_1\join \fr{a}_2$ and $d\in \fr{b}$ with $c\cdot d\leq e$. But $c\in \fr{a}_1\join \fr{a}_2$ means that there exist $c_1\in \fr{a}_1$, $c_2\in \fr{a}_2$ with $c_1\meet c_2\leq c$. By hypothesis $c_1\cdot d\in \fr{a}_1\bullet \fr{b}\subseteq \fr{c}$ and $c_2\cdot d\in \fr{a}_2\bullet \fr{b}\subseteq \fr{c}$. Because $\fr{c}$ is closed under meets, $(c_1\cdot d)\meet (c_2\cdot d)\in \fr{c}$. Since multiplication distributes over meet in $\mtl$- and $\gmtl$-algebras, this gives that $(c_1\meet c_2)\cdot d\in \fr{c}$, whence since $\bullet$ is order-preserving by Lemma \ref{lem:bullet properties}(2)(c), $(c_1\meet c_2)\cdot d\leq c\cdot d\leq e$ is in $\fr{c}$. This shows that $(\fr{a}_1\join \fr{a}_2)\cdot \fr{b}\subseteq \fr{c}$.

Now $\fr{a}_1\join \fr{a}_2$ need not be prime, but Lemma \ref{lem:filter to prime filter}(2) provides that there exists a prime filter $\fr{p}$ such that $\fr{a}_1\join \fr{a}_2\subseteq \fr{p}$ and $\fr{p}\bullet \fr{b}\subseteq \fr{c}$. We must then have $a\meet b\in \fr{p}$ and $\fr{p}\cdot \fr{b}\subseteq \fr{c}$, so $a\meet b\in \fr{b}\res \fr{c}$ as claimed. This shows that $\fr{b}\res \fr{c}$ is a filter.

It now suffices to show that $\fr{b}\res \fr{c}$ is prime. Let $a\join b\in y\res z$. Then there exists $\fr{a}\in\Xd{A}$ with $a\join b\in \fr{a}$ and $\fr{a}\bullet \fr{b}\subseteq \fr{c}$. But $\fr{a}$ is prime, whence $a\join b\in \fr{a}$ implies $a\in \fr{a}$ or $b\in \fr{a}$. This shows that $a\in \fr{b}\res \fr{c}$ or $b\in \fr{b}\res\fr{c}$, which shows that $\fr{b}\res \fr{c}$ is prime. Observe that $\fr{b}\res\fr{c}\subseteq\fr{c}\neq A$, whence if $\fr{c}\neq A$ we have that $\fr{b}\res\fr{c}$ is proper. Hence $\fr{b}\res\fr{c}\in \Xd{\bf A}$.

For the rest, suppose that $\fr{a}\bullet\fr{b}\subseteq \fr{c}$. Then for each $a\in \fr{a}$ we have that $a\in \fr{b}\res \fr{c}$ by definition, so $\fr{a}\subseteq \fr{b}\res \fr{c}$. On the other hand, suppose that $\fr{a}\subseteq \fr{b}\res \fr{c}$. We will show that $\fr{a}\bullet \fr{b}\subseteq \fr{c}$. Since $\bullet$ is order-preserving and commutative by Lemma \ref{lem:bullet properties}(2), so we immediately obtain $\fr{a}\bullet\fr{b}\subseteq \fr{b}\bullet (\fr{b}\res \fr{c})$. Let $c\in \fr{b}\bullet (\fr{b}\res \fr{c})$. Then there exist $a\in \fr{b}$, $b\in \fr{b}\res \fr{c}$ such that $a\cdot b\leq c$. But $b\in \fr{b}\res \fr{c}$ provides that there exists $\fr{w}\in\Xd{A}$ with $\fr{w}\cdot \fr{b}\subseteq \fr{c}$ and $b\in \fr{w}$. This gives that $a\cdot b\in \fr{b}\bullet \fr{w}\subseteq \fr{c}$, so as $\fr{c}$ is upward-closed we have $c\in \fr{c}$. It follows that $\fr{a}\cdot \fr{b}\subseteq \fr{b}\bullet (\fr{b}\res \fr{c})\subseteq \fr{c}$. This yields that
$$\fr{a}\bullet\fr{b}\subseteq\fr{c} \text{ if and only if } \fr{a}\subseteq\fr{b}\res\fr{c}$$
This completes the proof.
\end{proof}

Because an abstract object ${\bf S}$ of $\mtl^\tau$ (respectively, $\gmtl^\tau$) is order-isomorphic to $\Xd{\bf A}$ for some $\mtl$-algebra (respectively, $\gmtl$-algebra) ${\bf A}$ and this isomorphism preserves $\bullet$, we immediately obtain the following.

\begin{corollary}
Let ${\bf S}$ be an object of $\mtl^\tau$ or $\gmtl^\tau$, and suppose that ${\sf y},{\sf z}\in S$ are such that there exists ${\sf x}\in S$ with $R({\sf x},{\sf y},{\sf z})$. Then there is a greatest such ${\sf x}$, which we denote by ${\sf y}\res{\sf z}$. Moreover, ${\sf x}\bullet{\sf y}\leq {\sf z}$ if and only if ${\sf x}\leq {\sf y}\res {\sf z}$.
\end{corollary}

In considering multiplication on objects of $\mtl^\tau$ and $\gmtl^\tau$, it is sometimes convenient to consider the \emph{Routley star} (see \cite{RM1,RM2,Urquhart}). If ${\bf A}$ is an $\mtl$-algebra and $\fr{a}\in\Xd{A}$, we define
$$\fr{a}^*=\{a\in A : \neg a \not\in\fr{a}\}$$
It is easy to see that if $\fr{a}$ is a prime filter, then so is $\fr{a}^*$. Moreover, $^*$ is an order-reversing operation on prime filters.

\begin{lemma}\label{lem:routley star definition 1}
Let ${\bf A}$ be an $\mtl$-algebra and let $\fr{a}\in\Xd{A}$. Then $\fr{a}^*$ is the largest prime filter of ${\bf A}$ such that $\fr{a}\bullet\fr{a}^*\neq A$.
\end{lemma}

\begin{proof}
We will first show that $0\notin\fr{a}\bullet\fr{a}^*$. Toward a contradiction, suppose that $a\in\fr{a},b\in\fr{a}^*$ are such that $a\cdot b\leq 0$. Then by residuation $b\leq a\to 0=\neg a$, and since $\fr{a}^*$ is upward-closed it follows that $\neg a\in\fr{a}^*$. This implies that $\neg\neg a\notin\fr{a}$, and since $a\leq\neg\neg a$ and $\fr{a}$ is upward-closed we have $a\notin\fr{a}$. This contradicts the assumption that $a\in\fr{a}$, and hence $0\notin\fr{a}\bullet\fr{a}^*$. Therefore, $\fr{a}\bullet\fr{a}^*$ is proper.

Now suppose that $\fr{b}\not\subseteq\fr{a}^*$. Then there exists $b\in\fr{b}$ with $b\notin\fr{a}^*$, i.e., $\neg b\in\fr{a}$. It follows that $\neg b\cdot b\in\fr{a}\bullet\fr{b}$, so as $\neg b\cdot b = 0$ we have that $\fr{a}\bullet\fr{b}=A$. Thus $\fr{a}^*=\max\{\fr{b}\in\Xd{A} : \fr{a}\bullet\fr{b}\neq A\}$.
\end{proof}

\begin{corollary}\label{cor:routley star definition 2}
Let ${\bf S}$ be an object of $\mtl^{\tau}$, and let ${\sf x}\in S$. Then there exists a greatest ${\sf y}\in S$ such that $R({\sf x},{\sf y},{\sf z})$ for some ${\sf z}\in S$. Equivalently, there exists a greatest ${\sf y}\in S$ such that ${\sf x}\bullet{\sf y}$ is defined.
\end{corollary}

\begin{proof}
By the duality for $\mtl$, there exists an $\mtl$-algebra ${\bf A}$ such that ${\bf S}\cong \Xd{\bf A}$. Let $\alpha\colon {\bf S}\to \Xd{\bf A}$ be an isomorphism in $\mtl^\tau$. Then $\alpha({\sf x})$ is a prime filter of ${\bf A}$, and by Lemma \ref{lem:routley star definition 1} we have that $\alpha({\sf x})^*$ is the greatest element of $\Xd{\bf A}$ multiplying with $\alpha({\sf x})$ to give a proper filter. Then $R^\Xd{\bf A}(\alpha({\sf x}),\alpha({\sf x})^*,\alpha({\sf x})\bullet\alpha({\sf x})^*)$, whence since $\alpha^{-1}$ is an isomorphism (following from the fact that $\alpha$ is), we have that $R({\sf x},\alpha^{-1}(\alpha({\sf x})^*),\alpha^{-1}(\alpha({\sf x})\bullet\alpha({\sf x})^*))$. Now suppose that ${\sf y}\in S$ is such that there exists ${\sf z}\in S$ with $R({\sf x},{\sf y},{\sf z})$. Then $R^\Xd{\bf A}(\alpha({\sf x}),\alpha({\sf y}),\alpha({\sf z}))$, so by definition $\alpha({\sf x})\bullet\alpha({\sf y})\subseteq\alpha({\sf z})$. In particular, this means that $\alpha({\sf x})\bullet\alpha({\sf y})\neq A$, so by Lemma \ref{lem:routley star definition 1} it follows that $\alpha({\sf y})\subseteq\alpha({\sf x})^*$. Because $\alpha^{-1}$ is an order-isomorphism, we then have that ${\sf y}\leq\alpha^{-1}(\alpha({\sf x})^*)$. This shows that $\alpha^{-1}(\alpha({\sf x})^*) = \max\{{\sf y}\in S : \exists {\sf z}\in S, R({\sf x},{\sf y},{\sf z})\}$ as desired.
\end{proof}

In light of the previous corollary, for an abstract object ${\bf S}$ of $\mtl^\tau$ we define for any ${\sf x}\in S$,
$${\sf x}^*:=\max\{{\sf y}\in S : \exists {\sf z}\in S, R({\sf x},{\sf y},{\sf z})\}.$$
The operation just defined provides a convenient language for discussing phenomena in $\mtl^\tau$, and will be used extensively later.

\section{Representation of dual spaces by filter pairs}\label{sec:filterpairs}
As an initial step toward dualizing the quadruple construction of Section \ref{sec:quadruples}, we develop technical material relating the prime filters of an $\sbp$-algebra ${\bf A}$ to those of $\bool{\bf A}$ and $\rad{\bf A}$. Given an $\sbp$-algebra ${\bf A}$ and $\fr{a}\in\Xd{\bf A}$, the sets $\fr{a}\cap\bool{\bf A}$ and $\fr{a}\cap\rad{\bf A}$ are an ultrafilter of the Boolean skeleton of ${\bf A}$ and a generalized prime filter of the radical of ${\bf A}$, respectively. Conversely, for each $\sbp$-algebra ${\bf A}$ we define a collection $\Da{A}\subseteq \Xd{\bool{\bf A}}\times\Xd{\rad{\bf A}}$ as follows.
\begin{definition} $(\fr{u},\fr{x})\in\Da{A}$ if and only if
\begin{equation}\label{eq:da}\forall (u,x)\in\bool{\bf A}\times\rad{\bf A}, u\join x\in\fr{x} \text{ implies } u\in\fr{u}\text{ or }x\in\fr{x}.\end{equation}
$\Da{A}$ may be endowed with the obvious product order, i.e., by $(\fr{u},\fr{x})\subseteq (\fr{v},\fr{y})$ if and only if $\fr{u}\subseteq\fr{v}$ and $\fr{x}\subseteq\fr{y}$. In light of the fact that $\fr{u}$ and $\fr{v}$ are ultrafilters, $\fr{u}\subseteq\fr{v}$ entails that $\fr{u}=\fr{v}$. Thus $(\fr{u},\fr{x})\subseteq (\fr{v},\fr{y})$ if and only if $\fr{u}=\fr{v}$ and $\fr{x}\subseteq\fr{y}$. 
\end{definition}
Notice that for every prime filter $\fr{x}$ of $\delta[\rad{\bf A}]$, where $\delta\colon\rad{\bf A}\to\rad{\bf A}$ is the wdl-admissible map defined by $\delta(x)=\neg\neg x$, we have
$$\delta^{-1}[\fr{x}] = \rm{max}\{\fr{y} \in \Xd{\rad{{\bf A}}} : \delta[\fr{y}] = \fr{x}\}.$$
Indeed, it is easy to check that $\delta^{-1}[\fr{x}]$ is a prime filter, and any other prime filter $\fr{y}$ such that $\delta[\fr{y}] = \fr{x}$ is contained in $\delta^{-1}[\fr{x}]$.
\begin{definition}
Let us define: $\Da{A}^\partial = \{(\fr{u}, \fr{y}) \in \Xd{\bool{\bf A}}\times\Xd{\delta[\rad{\bf A}]}: (\fr{u},\delta^{-1}[\fr{y}])\in\Da{A} \mbox{ and } \delta^{-1}[\fr{y}]\neq\rad{\bf A}\}.$ \end{definition}
We will show that the dual space of an $\sbp$-algebra ${\bf A}$ may be realized as a structure constructed on a disjoint union of $\Da{A}$ with $\Da{A}^\partial$. For clarity of our exposition, we will decorate the members of $\Da{A}^\partial$ with an initial $+$, so that we may write $+(\fr{u}, \fr{y}) \in \Da{A}^\partial$ iff 
\begin{equation}\label{eq:Fdual}(\fr{u}, \fr{y}) \in \Xd{\bool{\bf A}}\times\Xd{\delta[\rad{\bf A}]}, (\fr{u},\delta^{-1}[\fr{y}])\in\Da{A} \mbox{ and } \delta^{-1}[\fr{y}]\neq\rad{\bf A}.\end{equation}
The notation $+$ is motivated by the fact that the members of $\Da{A}^\partial$ will represent an ``upper'' or ``positive'' portion of the dual space of ${\bf A}$, as is evident from the definition we make presently. 
\begin{definition}\label{def:fdual}Set $\Da{A}^{\bowtie} = \Da{A}\dot\cup\Da{A}^\partial$. We endow $\Da{A}^{\bowtie}$ with a partial order $\sqsubseteq$ given by $\fr{p}\sqsubseteq \fr{q}$ if and only if
\begin{enumerate}
\item $\fr{p}=(\fr{u},\fr{x})$ and $\fr{q}=(\fr{v},\fr{y})$ for some $(\fr{u},\fr{x}),(\fr{v},\fr{y})\in\Da{A}$ with $(\fr{u},\fr{x})\subseteq (\fr{v},\fr{y})$,
\item $\fr{p}=+(\fr{u}',\fr{x}')$ and $\fr{q}=+(\fr{v}',\fr{y}')$ for some $+(\fr{u}',\fr{x}'),+(\fr{v}',\fr{y}')\in\Da{A}^\partial$ with $(\fr{v}',\fr{y}')\subseteq (\fr{u}',\fr{x}')$, or
\item $\fr{p}=(\fr{u},\fr{x})$ and $\fr{q}=+(\fr{v},\fr{y})$ for some $(\fr{u},\fr{x})\in\Da{A}$, $(\fr{v},\fr{y})\in\Da{A}^\partial$ with $\fr{u}=\fr{v}$.
\end{enumerate}
\end{definition}
Intuitively, Condition \ref{eq:da} is a form of \emph{external primality}. Just as prime filters correspond to points in the dual space of a distributive lattice, filter pairs satisfying external primality (along with their ``positive'' companions) correspond to points in the dual quadruples construction to follow. An abstract treatment of external primality will prove important for this purpose, and we next turn to providing such an abstract description. 
 
Let ${\bf A}$ be an $\sbp$-algebra, and suppose that $\fr{a}\in\Xd{A}$. Then each $u\in\bool{\bf A}$ satisfies $u\join\neg u = 1\in\fr{a}$, and by primality either $u\in\fr{a}$ or $\neg u\in\fr{a}$. This entails that each $\fr{a}\in\Xd{A}$ contains an ultrafilter of $\bool{\bf A}$, and since $\fr{a}$ is a proper filter this must be the only ultrafilter of $\bool{\bf A}$ contained in $\fr{a}$. 

\begin{definition}For each $\fr{a}\in\Xd{A}$ we hence denote by $\fr{u}_\fr{a}$ the unique ultrafilter of $\bool{\bf A}$ contained in $\fr{a}$, and call it the \emph{ultrafilter of $\fr{a}$}. Furthermore, we say that $\fr{u}\in\bool{\bf A}$ \emph{fixes} $\fr{x}\in\Xd{\rad{A}}$ if there exists $\fr{a}\in\Xd{A}$ such that $\fr{u}\subseteq\fr{a}$ and $\fr{x}=\fr{a}\cap\rad{\bf A}$. 
\end{definition}
Notice that for each $\fr{a}\in\Xd{\bf A}$ we trivially have that $\fr{u}_\fr{a}$ fixes the generalized prime filter $\fr{a}\cap\rad{\bf A}$. 
The following lemma explains this terminology.

\begin{lemma}\label{lem:fixedpoints}
Let ${\bf A}$ be an $\sbp$-algebra, and define for each $u\in\bool{\bf A}$ a map $\mu_u\colon \Xd{\rad{\bf A}}\to\Xd{\rad{\bf A}}$ by $\mu_u(\fr{x})=\{x\in \rad{\bf A} : u\join x\in\fr{x}\}=\upnu_u^{-1}[\fr{x}]$. Then
\begin{enumerate}
\item $\fr{a}\cap\rad{\bf A}$ is a generalized prime filter of $\rad{\bf A}$ and is fixed by each of the maps $\mu_u$ for $u\notin\fr{u}_\fr{a}$.
\item Conversely, if $\fr{x}=\fr{a}\cap\rad{\bf A}$ is proper and $\fr{u}$ is an ultrafilter of $\bool{\bf A}$ such that $\fr{x}$ is fixed by each $\mu_u$, $u\notin\fr{u}$, then $\fr{u}\subseteq\fr{a}$ and thus $\fr{u}=\fr{u}_\fr{a}$.
\end{enumerate}
\end{lemma}

\begin{proof}
For the first claim, it is obvious that $\fr{a}\cap\rad{\bf A}$ is a generalized prime filter of $\rad{\bf A}$ since both $\fr{a}$ and $\rad{\bf A}$ are upward closed, closed under meets, and given any $a,b\in\rad{\bf A}$ with $a\join b\in\fr{a}$, either $a\in\fr{a}\cap\rad{\bf A}$ or $b\in\fr{a}\cap\rad{\bf A}$. For the rest, let $u\notin\fr{u}_\fr{a}$ and set $\fr{x}=\fr{a}\cap\rad{\bf A}$. If $x\in\fr{x}$, then $x\leq u\join x$ gives $u\join x\in\fr{x}$ and hence $x\in\mu_u(\fr{x})$. This provides $\fr{x}\subseteq \mu_{u}(\fr{x})$. For the reverse inclusion, suppose that $x\in\mu_u(\fr{x})$. Then $u\join x\in\fr{x}$, so since $\fr{x}\subseteq\fr{a}$ we have $u\join x\in\fr{a}$. But $\fr{a}$ is a prime filter of ${\bf A}$, so this implies $u\in\fr{a}$ or $x\in\fr{a}$. The former is impossible since $u\notin\fr{u}_\fr{a}$, so $x\in\fr{a}$. Since $x$ was chosen from $\rad{\bf A}$, this implies that $x\in\fr{a}\cap\rad{\bf A}$, giving the result.

For the second claim, let $u\in\fr{u}$. Then $\neg u\notin\fr{u}$ since ultrafilters are proper, so $\fr{x}$ is fixed by the map $\mu_{\neg u}$ by hypothesis. If $u\notin\fr{a}$, then $\neg u\in\fr{a}$ by the primality of $\fr{a}$. For each $x\in\rad{\bf A}$ we have $\neg u,x\leq \neg u\join x$, and because both $\rad{\bf A}$ and $\fr{a}$ are upward-closed this gives $\neg u\join x\in\fr{a}\cap\rad{\bf A}=\fr{x}$. Since $\fr{x}$ is fixed by $\mu_{\neg u}$, this implies that $x\in\mu_{\neg u}(\fr{x})=\fr{x}$. It follows that $\rad{\bf A}\subseteq\fr{x}$, contradicting the assumption that $\fr{x}$ is proper. Thus $u\in\fr{a}$, yielding $\fr{u}\subseteq\fr{a}$. Since the ultrafilter of $\fr{a}$ is unique, it follows that $\fr{u}=\fr{u}_\fr{a}$.
\end{proof}

The choice of terminology above would be more evocative if ultrafilters were replaced by their complements in $\bool{\bf A}$: An ultrafilter $\fr{u}$ fixes a proper $\fr{x}\in\Xd{\rad{\bf A}}$ if and only if $\fr{x}$ is a fixed point of each of the maps $\mu_u$ for $u\in\fr{u}^\comp$. However, working with prime filters rather than prime ideals is desirable thanks to other considerations.

The notion of an ultrafilter of $\bool{\bf A}$ fixing a prime filter of $\rad{\bf A}$ will be fundamental to dualizing the construction of Section \ref{sec:quadruples}, and we therefore study it in detail.

\begin{lemma}\label{lem:allfixed}
Let $\fr{x}\in\Xd{\rad{\bf A}}$. Then there exists an ultrafilter $\fr{u}$ of $\bool{\bf A}$ such that $\fr{u}$ fixes $\fr{x}$.
\end{lemma}

\begin{proof}
We must show that there exists a prime filter $\fr{a}$ of ${\bf A}$ such that $\fr{u}\subseteq\fr{a}$ and $\fr{x}=\fr{a}\cap\rad{\bf A}$. Set $\fr{i}=\rad{\bf A}\setminus\fr{x}$. Then $\fr{i}$ is trivially an ideal of $\rad{\bf A}$. Moreover, its down-set $\downset\fr{i} = \{a\in A : a\leq i\text{ for some }i\in\fr{i}\}$ considered in ${\bf A}$ may readily be seen to be an ideal of ${\bf A}$. Also, $\fr{x}$ is a filter of ${\bf A}$, and $\fr{x}\cap\downset\fr{i}=\emptyset$ by construction. It follows from the prime ideal theorem for distributive lattices that there exists $\fr{a}\in\Xd{\bf A}$ such that $\fr{a}\cap\downset\fr{i}=\emptyset$ and $\fr{x}\subseteq\fr{a}$. Then the ultrafilter $\fr{u}_\fr{a}$ of $\fr{a}$ fixes $\fr{x}$.
\end{proof}

A given $\fr{x}\in\Xd{\rad{\bf A}}$ may be fixed by many ultrafilters. As a stark example of this, it turns out that $\rad{\bf A}$ itself is fixed by \emph{every} ultrafilter of $\bool{\bf A}$.

\begin{lemma}\label{lem:radfixed}
Let $\fr{u}$ be an ultrafilter of $\bool{\bf A}$. Then there exists $\fr{a}\in\Xd{\bf A}$ such that $\fr{u}\subseteq\fr{a}$ and $\rad{\bf A}\subseteq\fr{a}$, whence $\fr{u}$ fixes $\rad{\bf A}$.
\end{lemma}

\begin{proof}
Let $\fr{f}$ be the filter of ${\bf A}$ generated by $\fr{u}\cup\rad{\bf A}$. We claim that $\fr{f}$ is proper. Indeed, suppose on the contrary that $0\in\fr{f}$. Then there exists $u\in\fr{u}$ and $x\in\rad{\bf A}$ such that $0=u\meet x$. It follows that $u\cdot x \leq 0$, and by residuation $x\leq u\to 0 = \neg u$. Since $\rad{\bf A}$ is upward-closed, this entails that $\neg u\in\rad{\bf A}$. But the only complemented element in $\rad{\bf A}$ is $1$, and hence $\neg u = 1$. This implies that $u=0$, which contradicts the assumption that $\fr{u}$ is a (necessarily proper) ultrafilter. Hence $\fr{f}$ is proper, and may therefore be extended to a prime filter $\fr{a}$ of ${\bf A}$. By construction $\fr{u}\subseteq\fr{a}$ and $\rad{\bf A}\subseteq\fr{a}$, proving the result.
\end{proof}

The following links external primality to the concept of an ultrafilter fixing a radical filter.

\begin{lemma}\label{lemmafixes}
Let ${\bf A}$ be an $\sbp$-algebra and let $\fr{x}\in\Xd{\scrR({\bf A})}$. Then $(\mathfrak{u},\mathfrak{x})\in \Da{A}$ if and only if $\mathfrak{u}$ fixes $\mathfrak{x}$.
\end{lemma}

\begin{proof}
Suppose first that $(\mathfrak{u},\mathfrak{x})\in \Da{A}$. We must show that $\mathfrak{u}$ fixes $\mathfrak{x}$. Let $u\notin\mathfrak{u}$. it suffices to show that $\mu_u(\mathfrak{x})\subseteq\mathfrak{x}$ since the reverse inclusion always holds, so let $x\in\mu_u(\mathfrak{x})$, then $u\join x\in\mathfrak{x}$. The fact that $(\mathfrak{u},\mathfrak{x})\in \Da{A}$ gives that $u\in\mathfrak{u}$ or $x\in\mathfrak{x}$. But $u\notin\mathfrak{u}$ by assumption, so $x\in\mathfrak{x}$. It follows that $\mu_u(\mathfrak{x})=\mathfrak{x}$ for every $u\notin\mathfrak{x}$. If $\fr{x}$ is a proper filter, then $\mathfrak{u}$ fixes $\mathfrak{x}$ by Lemma \ref{lem:fixedpoints}(2). On the other hand, if $\fr{x}=\rad{\bf A}$, then $\fr{u}$ fixes $\fr{a}$ by Lemma \ref{lem:radfixed}. The result holds in either case.

For the converse, suppose that $\mathfrak{u}$ fixes $\mathfrak{x}$. That $\mathfrak{u}$ and $\mathfrak{x}$ are nonempty prime filters of the Boolean skeleton and the radical, respectively, holds by hypothesis. Let $u\in\scrB({\bf A})$ and $x\in\scrR({\bf A})$ with $u\join x\in\mathfrak{x}$, and suppose that $u\notin\mathfrak{u}$. Then the assumption that $\mathfrak{u}$ fixes $\mathfrak{x}$ implies that $\mathfrak{x}=\mu_u (\mathfrak{x})$ by Lemma \ref{lem:fixedpoints}(1). We have $u\join x\in\mathfrak{x}$ implies that $x\in\mu_u(\mathfrak{x})$ by definition, so it follows that $x\in\mathfrak{x}$. This entails that $u\in\mathfrak{u}$ or $x\in\mathfrak{x}$, giving Condition \ref{eq:da} and thus the result.
\end{proof}

For each $\fr{x}\in\Xd{\rad{\bf A}}$, Lemma \ref{lem:allfixed} shows that $\{\fr{u} : \fr{u}\text{ fixes }\fr{x}\}\neq\emptyset$. It follows that $\fr{f}_\fr{x} = \cap \{\fr{u} : \fr{u}\text{ fixes }\fr{x}\}$ is a nonempty filter of $\bool{\bf A}$ (being the intersection of nonempty filters). This filter is an ultrafilter if and only if there is a unique ultrafilter $\fr{u}$ fixing $\fr{x}$ (i.e., $\fr{f}_\fr{x}$ itself).

\begin{lemma}\label{lem:ultraintersect}
Let $u\notin\fr{f}_\fr{x}$. Then $\mu_u$ fixes $\fr{x}$.
\end{lemma}

\begin{proof}
Because $u\notin\fr{f}_\fr{x}$ there exists an ultrafilter $\fr{u}$ of $\bool{\bf A}$ such that $\fr{u}$ fixes $\fr{x}$ and $u\notin\fr{u}$. Then $\mu_u(\fr{x})=\fr{x}$ by Lemma \ref{lem:fixedpoints}, which gives the result.
\end{proof}

The filter $\fr{f}_\fr{x}$ characterizes exactly which ultrafilters fix $\fr{x}$.

\begin{lemma}\label{lem:ultrachar}
Let $\fr{u}$ be an ultrafilter of $\bool{\bf A}$ and let $\fr{x}\in\Xd{A}$. Then $\fr{u}$ fixes $\fr{x}$ if and only if $\fr{f}_\fr{x}\subseteq\fr{u}$.
\end{lemma}

\begin{proof}
For the forward implication, suppose that $\fr{u}$ fixes $\fr{x}$. Then $\fr{f}_\fr{x}\subseteq\fr{u}$ since $\fr{f}_\fr{x}$ is an intersection of a set of ultrafilters containing $\fr{u}$.

For the reverse implication, note that if $\fr{x}=\rad{\bf A}$, then $\fr{u}$ trivially fixes $\fr{x}$ by Lemma \ref{lem:radfixed} (i.e., since every ultrafilter fixes the entire radical). If $\fr{x}\neq\rad{\bf A}$, then Lemma \ref{lem:allfixed} provides there exists $\fr{a}\in\Xd{A}$ such that $\fr{x}=\fr{a}\cap\rad{\bf A}$. Let $u\notin\fr{u}$. Then $u\notin\fr{f}_\fr{x}$ as $\fr{f}_\fr{x}\subseteq\fr{u}$, and by Lemma \ref{lem:ultraintersect} we have that $\mu_u$ fixes $\fr{x}$. It follows that $\fr{x}$ is fixed by each map $\mu_u$ for $u\notin\fr{u}$, and by Lemma \ref{lem:fixedpoints}(2) this yields that $\fr{u}\subseteq\fr{a}$. It follows that $\fr{u}$ fixes $\fr{x}$ as desired.
\end{proof}

The above lemma provides an intuitively-simple description of when $\fr{x}\in\Xd{\rad{\bf A}}$ is fixed by an ultrafilter $\fr{u}$: The ultrafilters fixing $\fr{x}$ are exactly those that extend $\fr{f}_\fr{x}$. We will now show how pairs $(\mathfrak{u}, \fr{x})$, where $\mathfrak{u}$ fixes $\fr{x}$, generate prime filters of \sbp-algebras.

\begin{proposition}\label{prop:prime}
Let $\mathfrak{u} \in \Xd{\scrB({\bf A})}$, $\mathfrak{x} \in \Xd{\scrR({\bf A})}$. If $\mathfrak{u}$ fixes $\mathfrak{x}$, then the generated filter $\mathfrak{p} = \langle \mathfrak{u} \cup \mathfrak{x} \rangle$ is prime. Moreover, $\mathfrak{p} \cap \scrB({\bf A})= \fr{u}$ and $\mathfrak{p} \cap \scrR({\bf A})= \fr{x}$. 
\end{proposition}
\begin{proof}
Since $\mathfrak{u}$ and $\mathfrak{x}$ are closed under meets, we have that 
$$\mathfrak{p} = \{ a \in A \mid u \land x \leq a \mbox{ for some } u \in \mathfrak{u}, x \in \mathfrak{x}\}$$
We will prove that $\mathfrak{p}$ is a prime filter of ${\bf A}$. Suppose that $a_{1} \lor a_{2} \in \mathfrak{p}$, where $a_{1} = (u_{1} \lor \neg x_{1}) \land (\neg u_{1} \lor x_{1})$ and $a_{2} =  (u_{2} \lor \neg x_{2}) \land (\neg u_{2} \lor x_{2})$ for some $u_1,u_2\in\bool{\bf A}$ and $x_1,x_2\in\rad{\bf A}$. We shall prove that $a_{1} \in \mathfrak{p}$ or $a_{2} \in \mathfrak{p}$, and also that $\mathfrak{p}$ is proper. 
Notice that $a_{1} \lor a_{2} \in \mathfrak{p}$ means $u \land x \leq a_{1} \lor a_{2}$, for some $u \in \mathfrak{u}, x \in \mathfrak{x}$. Using distributivity, we can write
\begin{align*}a_{1} \lor a_{2} = ((u_{1} \lor u_{2}) \lor (\neg x_{1} \lor \neg x_{2})) \land ((u_{1} \lor \neg u_{2}) \lor x_{2}) \land ((\neg u_{1} \lor u_{2}) \lor x_{1})\\ \land ((\neg u_{1} \lor \neg u_{2}) \lor (x_{1} \lor x_{2})).\end{align*}
Thus we get that $u \land x$ is bounded above by each of $(u_{1} \lor u_{2}) \lor (\neg x_{1} \lor \neg x_{2})$, $(u_{1} \lor \neg u_{2}) \lor x_{2}$, $(\neg u_{1} \lor u_{2}) \lor x_{1}$, and $(\neg u_{1} \lor \neg u_{2}) \lor (x_{1} \lor x_{2})$.
Consider the first one: $u \land x \leq (u_{1} \lor u_{2}) \lor (\neg x_{1} \lor \neg x_{2})$. This is possible if and only if $u \leq u_{1} \lor u_{2}$. To see this, note that since ${\bf A}$ isomorphic to $\scrB({\bf A}) \otimes_{e}^{\delta} \scrR({\bf A})$ (see Section \ref{sec:quadruples}), if we denote by $\lambda_{B}$ the isomorphism from $\scrB({\bf A})$ to $\scrB(\scrB({\bf A}) \otimes^{\delta}_{e} \scrR({\bf A}))$ and $\lambda_{R}$ the isomorphism from $\scrR({\bf A})$ to $\scrR(\scrB({\bf A}) \otimes_{e}^{\delta} \scrR({\bf A}))$, we get: \begin{align*} (\lambda_{B}(b) \sqcap \lambda_{R} (x)) &= [u, 1] \sqcap [1, x] = [u, \neg u \lor x],\\
\lambda_{B} (u_{1} \lor u_{2}) \sqcup \neg \lambda_{R}(x_{1} \land x_{2}) &= [u_{1} \lor u_{2}, 1] \sqcup [0, x_{1} \land x_{2}] \\&= [u_{1} \lor u_{2}, (u_{1} \lor u_{2}) \lor (x_{1} \land x_{2})].\end{align*}
Now note that $[u, \neg u \lor x] \land [u_{1} \lor u_{2}, (u_{1} \lor u_{2}) \lor (x_{1} \land x_{2})] = [u \land (u_{1} \lor u_{2}), \bar{x}],$
where $\bar{x}\in \scrR({\bf A})$ is calculated via the operations recalled in Section \ref{sec:quadruples}.
Via the isomorphism, $u \land x \leq (u_{1} \lor u_{2}) \lor (\neg x_{1} \lor \neg x_{2})$ corresponds to $[u \land (u_{1} \lor u_{2}), \bar{x}] = [u, \neg u \lor x]$, and this holds, as can be shown via calculations, if and only if $u \leq u_{1} \lor u_{2}$. By the primality of $\fr{u}$, we then have that not both of $u_{1}\notin\fr{u}$, $u_{2} \notin \mathfrak{u}$ may hold.
Moreover, notice that $(u_{1} \lor \neg u_{2}) \lor x_{2}$, $(\neg u_{1} \lor u_{2}) \lor x_{1}$ and $(\neg u_{1} \lor \neg u_{2}) \lor (x_{1} \lor x_{2})$ are elements of the radical, but if $a \in \scrR({\bf A})$ with  $u \land x \leq a$, via residuation, we get $x \leq \neg u \lor a$. Hence $\neg u \lor a \in \mathfrak{x}$, and using Lemma \ref{lemmafixes} plus condition \ref{eq:da} we obtain that either $\neg  u \in \mathfrak{u}$ or $a \in \mathfrak{x}$. But since $u \in \mathfrak{u}$, $\neg u \notin \mathfrak{u}$. Thus $a \in \mathfrak{x}$. Applying this reasoning to the three terms above, plus condition \ref{eq:da}, we obtain that the following facts hold: $x_{2} \in \mathfrak{x} \mbox{ or } u_{1} \lor \neg u_{2} \in \mathfrak{u}$; $x_{1} \in \mathfrak{x} \mbox{ or } \neg u_{1} \lor u_{2} \in \mathfrak{u}$; $x_{1} \lor x_{2} \in \mathfrak{x} \mbox{ or } \neg u_{1} \lor \neg u_{2} \in \mathfrak{u}$. It is now easy to check that if $u_{1}, \neg u_{2} \in \mathfrak{u}$ then $a_{1} \in \mathfrak{p}$; if $\neg u_{1}, u_{2} \in \mathfrak{u}$ then $a_{2} \in \mathfrak{p}$; if $u_{1}, u_{2} \in \mathfrak{u}$ then in case $x_{1} \in \mathfrak{x}$ it is $a_{1} \in \mathfrak{p}$, otherwise $x_{2} \in \mathfrak{x}$ and $a_{2} \in \mathfrak{p}$. Thus, $a_{1} \in \mathfrak{p}$ or $a_{2} \in \mathfrak{p}$.
Moreover, $\mathfrak{p}$ is clearly proper if $\mathfrak{u}$ is proper, because $u \land x > 0$ for any $u \in \mathfrak{u}$ and $x \in \mathfrak{x}$, and thus $0 \notin \mathfrak{p}$. As $\mathfrak{p}$ is proper,  $\mathfrak{p} \cap \scrB({\bf A})= \fr{u}$. We have already seen that if $u \land x \leq a$ for $a \in \scrR({\bf A})$ implies $a \in \mathfrak{x}$, thus $\mathfrak{p} \cap \scrR({\bf A})= \fr{x}$. 
\end{proof}
Following the same lines of proof of Proposition \ref{prop:prime} we can prove the following useful technical lemma.
\begin{lemma}\label{cor:Rustar}
Set $R_{\fr{u}} := \langle \mathfrak{u} \cup \scrR({\bf A})\rangle$, and let $\delta\colon{\bf A}\to{\bf A}$ by defined by $\delta(a)=\neg\neg a$.
\begin{enumerate}
\item $R_{\mathfrak{u}} \in \Xd{\bf A}$ for every $\mathfrak{u} \in \Xd{\scrB({\bf A})}$.
\item If $\fr{x} \in \Xd{\scrR({\bf A})}$, $\delta[\fr{x}] \neq \delta[\scrR({\bf A})]$, and $\fr{u}$ fixes $\fr{x}$, then we have
\begin{equation}\label{eq:genrustar}\langle \mathfrak{u} \cup \mathfrak{x}\rangle^{*} = \{ a \in A \mid u \land \neg x \leq a, \mbox{ for some } u \in \mathfrak{u}, \neg\neg x \in \delta[\scrR({\bf A})] \setminus \delta[\mathfrak{x}]\}\end{equation}
\item Under the hypotheses of (2),
$$\langle \mathfrak{u} \cup \mathfrak{x}\rangle^{*} \cap \scrC({\bf A}) = \{\neg x : \neg\neg x \in \delta[\scrR({\bf A})] \setminus \delta[\mathfrak{x}]\}$$
and $\langle \mathfrak{u} \cup \mathfrak{x}\rangle \subseteq \langle \mathfrak{u} \cup \mathfrak{x}\rangle^{*}$.
\item  If $\fr{x} \in \Xd{\scrR({\bf A})}$ with $\delta[\fr{x}] = \delta[\scrR({\bf A})]$, then $\langle \mathfrak{u} \cup \mathfrak{x}\rangle^{*} = R_{\mathfrak{u}}.$
\end{enumerate}
\end{lemma} 
\begin{proof}
Since the radical is fixed by every ultrafilter of the Boolean skeleton, proof of (1) follows from Proposition \ref{prop:prime}. In order to prove (2), we check identity \ref{eq:genrustar}.
To prove the right-inclusion, let $u \land \neg x \leq a$ with $ u \in \mathfrak{u}, \neg\neg x \in  \delta[\scrR({\bf A})] \setminus \delta[\mathfrak{x}]$. Then $\neg a \leq \neg u \lor \neg\neg x$. It follows that $\neg a \notin \langle \mathfrak{u} \cup \mathfrak{x}\rangle$; otherwise we would have $\neg u \lor \neg\neg x \in \langle \mathfrak{u} \cup \mathfrak{x}\rangle$, but since $u \in \fr{u},\neg x \notin \fr{u}, \neg\neg x \notin \delta[\mathfrak{x}]\subseteq \fr{x}$ this would lead to contradiction. Hence $a \in \langle \mathfrak{u} \cup \mathfrak{x}\rangle^{*}$.

To prove the reverse inclusion, we again use the decomposition of the elements in terms of the Boolean skeleton and radical. Let $a \in \langle \mathfrak{u} \cup \mathfrak{x}\rangle^{*}$. Then $\neg a \notin \langle \mathfrak{u} \cup \mathfrak{x}\rangle$, where we write $a = (u \land x) \lor (\neg u \land \neg x)$ and $\neg a = (\neg u \land \neg\neg x) \lor (u \land \neg x)$ in accordance with the representation given in Section \ref{sec:quadruples}. If $u \in \fr{u}$, then we have that $a \geq  u \land x \geq u \land \neg y$ for any $y \in \scrR({\bf A})$ \cite{AguzFlamUgol}, and since $\delta[\fr{x}] \neq \delta[\scrR({\bf A})]$ there exists $z\in \scrR({\bf A})$ such that $\neg\neg z \notin\fr{x}$ and $a \geq  u \land \neg z$. Otherwise, if $u \notin \fr{u}$, $\neg u \in \fr{u}$ and since $\neg u \land \neg\neg x \leq \neg a \notin \langle \mathfrak{u} \cup \mathfrak{x}\rangle$, we get that $\neg\neg x \notin \fr{x}$. Thus $a \geq \neg u \land \neg x$ with $\neg\neg x \in\delta[\scrR({\bf A})] \setminus \delta[\mathfrak{x}]$. This proves identity \ref{eq:genrustar} and thus (2).

In order to prove $\langle \mathfrak{u} \cup \mathfrak{x}\rangle \subseteq \langle \mathfrak{u} \cup \mathfrak{x}\rangle^{*}$, let $a \in \langle \mathfrak{u} \cup \mathfrak{x}\rangle$. Then $u \land x \leq a$ for some $u \in \fr{u}, x \in \fr{x}$. Reasoning as before, there exists $\neg\neg x \in \delta[\scrR({\bf A})] \setminus \delta[\mathfrak{x}]$ and $u \land \neg x \leq u \land x \leq a$, thus $a \in  \langle \mathfrak{u} \cup \mathfrak{x}\rangle^{*}$.
Finally, $\langle \mathfrak{u} \cup \mathfrak{x}\rangle^{*} \cap \scrC({\bf A}) = \{\neg x : \neg\neg x \in \delta[\scrR({\bf A})] \setminus \delta[\mathfrak{x}]\}$ follows from the definition of the operator $^{*}$. This proves (3).

We now prove (4). The inclusion $\langle \mathfrak{u} \cup \mathfrak{x}\rangle^{*} \subseteq R_{\mathfrak{u}}$ can be proved again via calculations using the decomposition of the elements. For the other inclusion, let $a \in R_{\mathfrak{u}}$, with $u \land x \leq a$ for some $u \in \fr{u}, x \in \scrR({\bf A})$. Then $\neg a \leq \neg u \lor \neg x$. If $\neg a \in \langle \mathfrak{u} \cup \mathfrak{x}\rangle$, then by primality one of $\neg u$ or $\neg x$ is in $\langle \mathfrak{u} \cup \mathfrak{x}\rangle$. But $\neg u \notin \fr{u}$ and $\neg x \in \scrC({\bf A})$ (see Lemma \ref{lemma:corad}), and thus they are not in $\langle \mathfrak{u} \cup \mathfrak{x}\rangle$. Hence $\neg a \notin \langle \mathfrak{u} \cup \mathfrak{x}\rangle$, which means $a \in \langle \mathfrak{u} \cup \mathfrak{x}\rangle^{*}$. The claim follows.
\end{proof}

\begin{lemma}\label{lem:routleystar}
Let ${\bf A}$ be an $\sbp$-algebra and let $\fr{a}\in\Xd{\bf A}$. Then either $\fr{a}\subseteq\fr{a}^*$ or $\fr{a}^*\subseteq\fr{a}$.
\end{lemma}

\begin{proof}
Suppose that $\fr{a}\not\subseteq\fr{a}^*$. Then there exists $a\in\fr{a}$ with $a\notin\fr{a}^*$. This gives $\neg a\in\fr{a}$ by the definition of $^*$, so $a,\neg a\in\fr{a}$ and thus $a\meet\neg a\in\fr{a}$. In any $\sbp$-algebra, we have $a\meet\neg a\leq b\join\neg b$ for any elements $a$ and $b$. Taking $b\in\fr{a}^*$, this gives that $b\join\neg b\in\fr{a}$ since filters are upward closed. Since $\fr{a}$ is prime, this yields $b\in\fr{a}$ or $\neg b\in\fr{a}$. In the latter case, we would have $b\notin\fr{a}^*$, contradicting the fact that $b$ was chosen from $\fr{a}^*$. It follows that $b\in\fr{a}$, so $\fr{a}^*\subseteq\fr{a}$. This proves the claim.
\end{proof}

Observe that if $\fr{a},\fr{b}\in\Xd{A}$ and $\fr{a}\subseteq\fr{b}$, then $\fr{a}$ and $\fr{b}$ must contain the same ultrafilter. The collection of prime filters having a given ultrafilter is a natural notion in this setting, and we arrive at the following definition.
\begin{definition}For an ultrafilter $\fr{u}$ of the Boolean skeleton of an $\sbp$-algebra, we define $\site{u} = \{\fr{a}\in\Xd{A} : \fr{u}\subseteq\fr{a}\}$ and call $\site{u}$ the \emph{site of $\fr{u}$ in ${\bf A}$}. Whenever it is convenient, we may regard $\site{u}$ as a partially-ordered set under inclusion.
\end{definition}
%
Using Lemma \ref{lem:routleystar}, we obtain the following.

\begin{lemma}\label{lem:sameultrafiltersite}
Let $\fr{a}\in\Xd{A}$. Then $\fr{a}$ and $\fr{a}^*$ have the same ultrafilter, and hence $\site{u}$ is closed under $^*$ for each ultrafilter $\fr{u}$ of $\bool{\bf A}$.
\end{lemma}

\begin{proof}
By Lemma \ref{lem:routleystar}, either $\fr{a}\subseteq\fr{a}^*$ or $\fr{a}^*\subseteq\fr{a}$. Let $\fr{u}$ be the ultrafilter of the least of $\fr{a},\fr{a}^*$. Then $\fr{u}\subseteq\fr{a}$ and $\fr{u}\subseteq\fr{a}^*$. Since the ultrafilter contained in a prime filter of ${\bf A}$ is unique, this implies that $\fr{u}$ is the ultrafilter of each of $\fr{a}$ and $\fr{a}^*$.
\end{proof}

\begin{lemma}\label{lem:radicalstar}
Let ${\bf A}$ be an $\sbp$-algebra and let $\fr{a}\in\Xd{A}$. Then one of $\fr{a}$ or $\fr{a}^*$ contains $\rad{\bf A}$.
\end{lemma}

\begin{proof}
Let $a\in\rad{\bf A}\setminus\fr{a}$. Because $\neg a < a$ for each $a\in\rad{\bf A}$, this yields that $\neg a\notin\fr{a}$ (for if $\neg a\in\fr{a}$, then we would have $a\in\fr{a}$ because $\fr{a}$ is upward-closed). It follows that $a\in\fr{a}^*$, and this shows $\rad{\bf A}\setminus\fr{a}\subseteq\fr{a}^*$.

Now suppose that $\rad{\bf A}\not\subseteq\fr{a}$. Then $\rad{\bf A}\setminus\fr{a}\neq\emptyset$, so the comments above give that there exists $a\in\fr{a}^*$ with $a\notin\fr{a}$. Because Lemma \ref{lem:routleystar} provides that $\fr{a}\subseteq\fr{a}^*$ or $\fr{a}^*\subseteq\fr{a}$, it follows that $\fr{a}\subseteq\fr{a}^*$. Hence $\rad{\bf A}\setminus\fr{a}\subseteq\fr{a}^*$ and $\rad{\bf A}\cap\fr{a}\subseteq\fr{a}\subseteq\fr{a}^*$, so $\rad{\bf A}=(\rad{\bf A}\cap\fr{a})\cup (\rad{\bf A}\setminus\fr{a})\subseteq\fr{a}^*$, giving the result.
\end{proof}
\begin{lemma}\label{lemma:incl}
For every $\fr{a} \in \Xd{A}$, either $\fr{a} \subseteq \scrR_{\fr{u}_{\fr{a}}} \subseteq \fr{a}^{*}$ or $\fr{a}^{*} \subseteq \scrR_{\fr{u}_{\fr{a}}} \subseteq \fr{a}$.
\end{lemma}
\begin{proof}
From Lemma \ref{lem:routleystar}, either $\fr{a} \subseteq  \fr{a}^{*}$ or $\fr{a}^{*} \subseteq \fr{a}$. We consider the case $\fr{a}^{*} \subseteq \fr{a}$. By Lemma \ref{lem:radicalstar} we obtain that $\rad{\bf A} \subseteq \fr{a}$ and this implies that $R_{\fr{u}} = \langle \rad{\bf A} \cup \fr{u} \rangle \subseteq \fr{a}$. 

First consider the case in which $R_{\fr{u}} \not\subseteq \fr{a}^{*}$. We need to show that $\fr{a}^{*} \subseteq R_{\fr{u}}$, so let $a \in \fr{a}^{*}$. Then from identity \ref{eq:el}, $a = (u \land x) \lor (\neg u \land \neg x)$ for some $u\in\bool{\bf A}$, $x\in\rad{\bf A}$. Since $\fr{a}^{*}$ is prime, either $u \land x$ or $\neg u \land \neg x$ is in $\fr{a}^{*}$. If $u \notin \fr{u}$, we would have $\neg u \land \neg x \in \fr{a}^{*}$, so $\neg x \in \fr{a}^{*}$. But since $\fr{a}^{*}$ upwards closed and $\neg x \leq y$ for every $y \in \rad{\bf A}$ gives that $\rad{\bf A}\subseteq\fr{a}^*$. Thus $R_{\fr{u}} \subseteq \fr{a}^{*}$, a contradiction. Hence $u \in \fr{u}$ and $u \land x \in \fr{a}^{*}$. Since $u \land x \leq a$, and $u \land x \in R_{\fr{u}}$, we get $\fr{a}^{*}\subseteq R_{\fr{u}}$. 

Now consider the case where $R_{\fr{u}} \subseteq \fr{a}^{*}$. Let $x \in \scrC({\bf A})$. Then from Lemma \ref{lemma:corad}, $\neg x \in \rad{\bf A}\subseteq\scrR_{\fr{u}} \subseteq \fr{a}^{*} \subseteq \fr{a}$, which implies $x \notin \fr{a}^{*}$ and $\neg\neg x \notin \fr{a}$ from the definition of $*$ and the fact that $x\leq\neg\neg x$. Thus also $x \notin \fr{a}$. Let $a \in \fr{a}$. Again from identity \ref{eq:el} we get $a = (u \land y) \lor (\neg u\land \neg y)$ (where as always $u\in\bool{\bf A}$, $y\in\rad{\bf A}$), and then either $u \land y$ or $\neg u \land \neg y$ is in $\fr{a}$. But since $\neg y \in \scrC({\bf A})$, we have $\neg y \notin \fr{a}$ from the above. This means that $u \land y \in \fr{a}$, and thus in particular $u \in \fr{u}$ and $u\land y \in \scrR_{\fr{u}}$. Because $u\land y \leq a$, this gives $\scrR_{\fr{u}_{\fr{a}}} = \fr{a} = \fr{a}^{*}$.

We can prove in a completely analogous way that if $\fr{a} \subseteq \fr{a}^{*}$, then $\fr{a} \subseteq R_{\fr{u}_{\fr{a}}} \subseteq \fr{a}^{*}$.
\end{proof}
\begin{lemma}\label{lem:radical is star fixed}
$R_{\fr{u}} = R_{\fr{u}}^{*}$, for every $u \in \mathcal{S}(\scrB({\bf A}))$.
\end{lemma}
\begin{proof}
We prove first that $R_{\fr{u}} \subseteq  R_{\fr{u}}^{*}$. Let $a \in R_{\fr{u}}$. We want to show that $\neg a \notin R_{\fr{u}}$. Again using the decomposition result (identity \ref{eq:el}), we may write $a = (\neg u \lor x) \land (u \lor \neg x) =  (u \land x) \lor (\neg u \land \neg x)$ for some $u\in\bool{\bf A}$ and $x\in\rad{\bf A}$. Note that $\neg x \notin R_{\fr{u}}$ since in every $\sbp$-algebra, for every $u \in \scrB({\bf A}), x, y \in \scrR({\bf A}), \,u \land y \leq \neg x$ iff $u = 0$, as can be proven in $\scrB({\bf A}) \otimes_{e}^{\delta} \scrR({\bf A})$ by recalling that Boolean elements, elements of the radical, and elements of the coradical are respectively of the kind $[u, 1], [1, x], [0, y]$ (see \cite{AguzFlamUgol}). This implies that $u \land x \in R_{\fr{u}}$, and thus $u \in \fr{u}$. We have $\neg a = (u \land \neg x) \lor (\neg u \land \neg\neg x)$. Suppose that $\neg a \in R_{\fr{u}}$. Then either $u \land \neg x \in R_{\fr{u}}$ or $ \neg u \land \neg\neg x \in R_{\fr{u}}$. But $u \land \neg x \in R_{\fr{u}}$ implies $\neg x  \in R_{\fr{u}}$, which is a contradiction, and $ \neg u \land \neg\neg x \in R_{\fr{u}}$ implies $ \neg u \in R_{\fr{u}}$ which is again a contradiction since $u \in \fr{u}$. Thus, $\neg a \notin R_{\fr{u}}$, which implies that $R_{\fr{u}} \subseteq  R_{\fr{u}}^{*}$.

We now prove that $R_{\fr{u}}^{*} \subseteq  R_{\fr{u}}$. Let $a \in R_{\fr{u}}^{*}$. Then $\neg a \notin R_{\fr{u}}$. Again, we have $a = (u \land x) \lor (\neg u \land \neg x)$, $\neg a = (u \land \neg x) \lor (\neg u \land \neg\neg x)$ for some $u\in\bool{\bf A}$ and $x\in\rad{\bf A}$. If $u \notin \fr{u}$, we have $\neg u \in \fr{u}$ and thus, because $\neg\neg c \in \scrR({\bf A})$, $\neg u \land \neg \neg c \in R_{\fr{u}}$. But $\neg u \land \neg \neg x \leq \neg a$, implying $\neg a \in   R_{\fr{u}}$, a contradiction. It follows that $u \in \fr{u}$, and thus $u \land x \in R_{\fr{u}}$, $u \land x \leq a$ and $a \in R_{\fr{u}}$. This shows $R_{\fr{u}} = R_{\fr{u}}^{*}$.
\end{proof}

\begin{theorem}\label{th:dualiso}
Let ${\bf A}$ be a $\sbp$-algebra. Then
 $\mathcal{S}({\bf A})$ is order isomorphic to $\mathcal{F}_{\bf A}^{\bowtie}$.
\end{theorem}
\begin{proof}
Let $\alpha: \mathcal{S}({\bf A}) \to \mathcal{F}_{\bf A}^{\bowtie}$ be defined as follows:
$$
\alpha(\mathfrak{a}) = \left\{\begin{array}{ll}
 (\mathfrak{a} \cap \scrB({\bf A}),\, \mathfrak{a} \cap \scrR({\bf A})), &\mbox{ if } \mathfrak{a} \subseteq \mathfrak{a}^{*},\\
 +(\mathfrak{a}^{*} \cap \scrB({\bf A}), \delta[\fr{a}^*\cap\rad{\bf A}]) &\mbox{ otherwise},
\end{array}\right.
$$
where $\delta\colon A\to A$ is defined by $\delta(a)=\neg\neg a$ as usual. We will show that $\alpha$ defines an order isomorphism. We first show that it is a well-defined map. 
We only need prove that, given $\mathfrak{a} \in \mathcal{S}({\bf A})$, $\alpha(\mathfrak{a}) \in \mathcal{F}_{\bf A}^{\bowtie}$. By Lemma \ref{lem:routleystar}, either $\mathfrak{a} \subseteq \mathfrak{a}^{*}$ or $\mathfrak{a}^{*} \subset \mathfrak{a}$. Firstly, suppose that $\mathfrak{a} \subseteq \mathfrak{a}^{*}$. It is easy to see that $\mathfrak{a} \cap \bool{\bf A}$ is a prime filter of $\bool{\bf A}$, since both $\mathfrak{a}$ and $\bool{\bf A}$ are upward closed, closed under meets and moreover, given any $u, v \in \bool{\bf A}$ such that $u \lor v \in \mathfrak{a}$, then either $u \in \mathfrak{a} \cap \bool{\bf A}$ or $v \in \mathfrak{a} \cap \bool{\bf A}$. With exactly the same reasoning, we can prove that $\mathfrak{a} \cap \scrR({\bf A})$ is a prime filter of $\scrR({\bf A})$.
Let us now check condition \ref{eq:da}. Let  $u \in \bool{\bf A}$ and $x \in \scrR({\bf A})$ such that $u \lor x \in \mathfrak{a} \cap \scrR({\bf A})$. Since $\mathfrak{a}$ is prime,  either $u \in \mathfrak{a}$, which implies $u \in \mathfrak{a} \cap \bool{\bf A}$, or $x \in \mathfrak{a}$, in which case $x \in \mathfrak{a} \cap \scrR({\bf A})$. 
Now, if $\mathfrak{a}^{*} \subset \mathfrak{a}$, 
then $ \mathfrak{a}^{*} \cap \bool{\bf A}$  is a prime filter of $\bool{\bf A}$ and $\mathfrak{a}^{*} \cap \scrR({\bf A})$ is a prime filter of the radical, reasoning as before. Thus, it is easy to see that $\delta[\mathfrak{a}^{*} \cap \scrR({\bf A})]$ is a prime filter of $\delta[\scrR({\bf A})]$. We need to show that $(\mathfrak{a}^{*} \cap \bool{\bf A}, \delta^{-1}[\delta[\fr{a}^*\cap\rad{\bf A}]]) \in \mathcal{F}_{\bf A}$, in particular that given $u \lor x \in \delta^{-1}[\delta[\fr{a}^*\cap\rad{\bf A}]]$ either $u \in \mathfrak{a}^{*} \cap \bool{\bf A}$ or $x \in \delta^{-1}[\delta[\fr{a}^*\cap\rad{\bf A}]]$. If $u \lor x \in \delta^{-1}[\delta[\fr{a}^*\cap\rad{\bf A}]]$ then $\delta(u \lor x) = \neg\neg (u\lor x) = u \lor \neg\neg x \in \delta[\fr{a}^*\cap\rad{\bf A}] \subseteq \fr{a}^*\cap\rad{\bf A} \subseteq \fr{a}^*$ which is prime, thus either $u \in \fr{a}^{*}$ or $\delta(x) \in \fr{a}^{*}$. This means $u \in \mathfrak{a}^{*} \cap \bool{\bf A}$ or $\delta(x) \in \delta[\fr{a}^*\cap\rad{\bf A}]$ (i.e. $x \in \delta^{-1}[\delta[\fr{a}^*\cap\rad{\bf A}]]$), thus $(\mathfrak{a}^{*} \cap \bool{\bf A}, \delta^{-1}[\delta[\fr{a}^*\cap\rad{\bf A}]]) \in \mathcal{F}_{\bf A}$.


We shall now prove that $\alpha$ is a bijection.
We first prove surjectivity. Recall that $(\mathfrak{u}, \mathfrak{x}) \in \mathcal{F}_{\bf A}$ iff $\mathfrak{u}$ fixes $\mathfrak{x}$ by Lemma \ref{lemmafixes}. 
Let us consider $\fr{a} = \langle \fr{u} \cup \fr{x}\rangle$ which is prime from Proposition \ref{prop:prime}. Moreover $\mathfrak{a}\cap \scrB({\bf A}) = \mathfrak{u}$, $\mathfrak{a}\cap \scrR({\bf A}) = \mathfrak{x}$ and $\fr{a} \subseteq \fr{a}^{*}$ from Lemma \ref{cor:Rustar}. Thus $\alpha(\fr{a}) = (\fr{u}, \fr{x})$. Now let $+(\mathfrak{u}, \mathfrak{y}) \in \mathcal{F}_{\bf A}^{\partial}$. By definition, this means that if $\mathfrak{x} = \delta^{-1}[\mathfrak{y}]$, we have $(\mathfrak{u}, \mathfrak{x}) \in \mathcal{F}_{\bf A}$. Let $\fr{a} = \langle \fr{u} \cup \fr{x}\rangle^{*}$. Then by Lemma \ref{cor:Rustar} we get $\fr{a}^{*} \subseteq \fr{a}$ and from Lemma \ref{lem:sameultrafiltersite}, we have $\fr{a} \cap \scrB({\bf A}) = \fr{u}$.
It is easy to check via calculations that  $\delta[\fr{a}^*\cap\rad{\bf A}] = \neg (\scrC({\bf A}) \setminus \mathfrak{a}) $ for every $\fr{a} \in \mathcal{S}({\bf A})$, where $\scrC({\bf A})$ is the coradical of ${\bf A}$ (see Lemma \ref{lemma:corad}). 
Now, from Lemma \ref{cor:Rustar} we get that $\scrC({\bf A}) \setminus \mathfrak{a} = \{ \neg x: \neg\neg x \in \delta[\fr{y}] \cap \delta[\scrR({\bf A})]\}$ and then $\neg (\scrC({\bf A}) \setminus \mathfrak{a}) = \delta[\fr{y}]$. Thus,  $\alpha(\mathfrak{a})= +(\mathfrak{u}, \mathfrak{y})$ and
$\alpha$ is surjective.
It is easy to prove the injectivity of $\alpha$ using Equation \ref{eq:el}.

It remains to prove only that
$$ \mathfrak{a}_{1} \subseteq  \mathfrak{a}_{2} \mbox{ iff } \alpha(\mathfrak{a}_{1}) \leq \alpha(\mathfrak{a}_{2}) \mbox{ for any } \mathfrak{a}_{1},  \mathfrak{a}_{2} \in \Xd{\bf A}.$$
The fact that $ \mathfrak{a}_{1} \subseteq  \mathfrak{a}_{2}$ implies $ \alpha(\mathfrak{a}_{1}) \leq \alpha(\mathfrak{a}_{2})$ follows easily from the definition.
Let us now suppose that $\alpha(\mathfrak{a}_{1}) \leq \alpha(\mathfrak{a}_{2})$. 
We write $\mathfrak{u}_{1}= \mathfrak{a}_{1} \cap \bool{\bf A},\, \mathfrak{x}_{1} = \mathfrak{a}_{1} \cap \scrR({\bf A}), \mathfrak{y}_{1} = \delta[\mathfrak{a}_{1} \cap \scrR({\bf A})]$, and $\mathfrak{u}_{2}= \mathfrak{a}_{2} \cap \bool{\bf A},\, \mathfrak{x}_{2} = \mathfrak{a}_{2}\cap \scrR({\bf A}), \mathfrak{y}_{2} = \delta[\mathfrak{a}_{2} \cap \scrR({\bf A})]$. 
We distinguish four cases:
\begin{enumerate}
\item $\mathfrak{a}_{1} \subseteq \mathfrak{a}_{1}^{*}$, $\mathfrak{a}_{2} \subseteq \mathfrak{a}_{2}^{*}$.
Then $\mathfrak{a}_{1} \cap \scrR({\bf A}) \subseteq \mathfrak{a}_{2} \cap \scrR({\bf A})$. Let us prove that if $a \in \mathfrak{a}_{1}$ then $a \in \mathfrak{a}_{2}$. Via Equation \ref{eq:el}, $a =  (\neg u \lor x) \land (u \lor \neg x)$ for some $u\in\bool{\bf A}$ and $x\in\rad{\bf A}$. Thus, $\neg u \lor x, u \lor \neg x \in \mathfrak{a}_{1}$, which is prime. Thus, by $\neg u \lor x \in \mathfrak{a}_{1}$, we get that or $\neg u \in \mathfrak{u}_{1} \subseteq \mathfrak{u}_{2}$, or $x \in \mathfrak{x}_{1} \subseteq \mathfrak{x}_{2}$. In both cases $\neg u \lor x \in \mathfrak{a}_{2}$. By $u \lor \neg x \in \mathfrak{a}_{1}$ (which follows from $\neg x \notin \mathfrak{a}_{1}$ as a consequence of Lemma \ref{lemma:incl} and the fact that $\neg x \leq y$ for every $y \in \scrR({\bf A})$), we have that $u \in \mathfrak{u}_{1}\subseteq \mathfrak{u}_{2}$ and $u \lor \neg x \in \mathfrak{a}_{2}$.
\item $\mathfrak{a}_{1}^{*} \subseteq \mathfrak{a}_{1}$, $\mathfrak{a}_{2} \subseteq \mathfrak{a}_{2}^{*}$. By the definition of the order on $\mathcal{F}_{\bf A}^{\bowtie}$, this contradicts the fact that $\alpha(\mathfrak{a}_{1}) \leq \alpha(\mathfrak{a}_{2})$.
\item $\mathfrak{a}_{1} \subseteq \mathfrak{a}_{1}^{*}$, $\mathfrak{a}_{2}^{*} \subseteq \mathfrak{a}_{2}$. From the hypothesis it follows that $\fr{u}_{1} = \fr{u}_{2}$ and then using Lemma \ref{lemma:incl}, $\mathfrak{a}_{1} \subseteq R_{\fr u_{1}} \subseteq \mathfrak{a}_{2}$.
\item $\mathfrak{a}_{1}* \subseteq \mathfrak{a}_{1}$, $\mathfrak{a}_{2}* \subseteq \mathfrak{a}_{2}$. Again, since $\delta[\fr{a}^*\cap\rad{\bf A}] = \neg (\scrC({\bf A}) \setminus \mathfrak{a}) $ for every $\fr{a} \in \mathcal{S}({\bf A})$, we have $\neg (\scrC({\bf A}) \setminus \mathfrak{a}_{2}) \subseteq \neg(\scrC({\bf A}) \setminus \mathfrak{a}_{1})$. Notice that this implies that $\mathfrak{a}_{1} \cap \scrC({\bf A}) \subseteq \mathfrak{a}_{2} \cap \scrC({\bf A})$. Indeed, $\neg (\scrC({\bf A}) \setminus \mathfrak{a}_{2}) \subseteq \neg(\scrC({\bf A}) \setminus \mathfrak{a}_{1})$ implies $\neg \neg (\scrC({\bf A}) \setminus \mathfrak{a}_{2}) \subseteq \neg\neg(\scrC({\bf A}) \setminus \mathfrak{a}_{1})$ and since $\neg \neg (\scrC({\bf A}) \setminus \mathfrak{a}_{1,2}) = (\scrC({\bf A}) \setminus \mathfrak{a}_{1,2})$, we have that $\scrC({\bf A}) \setminus \mathfrak{a}_{2} \subseteq \scrC({\bf A}) \setminus \mathfrak{a}_{1}$, which implies $\mathfrak{a}_{1} \cap \scrC({\bf A}) \subseteq \mathfrak{a}_{2} \cap \scrC({\bf A})$.
We now consider again $a \in \mathfrak{a}_{1}$, where as usual we write $a = (\neg u \lor x) \land (u \lor \neg x)$ for $u\in\bool{\bf A}$ and $x\in\rad{\bf A}$. Then $\neg u \lor x, u \lor \neg x \in \mathfrak{a}_{1}$ by primality. 
Because any $x \in \scrR({\bf A})$ is both in $\mathfrak{a}_{1}$ and $\mathfrak{a}_{2}$, clearly $\neg u \lor x$ is in $\mathfrak{a}_{2}$. By $u \lor \neg x \in \mathfrak{a}_{1}$, we have $u \in \mathfrak{u}_{1}\subseteq \mathfrak{u}_{2}$, or $\neg x \in \mathfrak{a}_{1} \cap \scrC({\bf A}) \subseteq \mathfrak{a}_{2} \cap \scrC({\bf A})$. In both cases $u \lor \neg x \in \mathfrak{a}_{2}$. Hence $a \in \mathfrak{a}_{2}$, and the proof is settled. 
\end{enumerate}
\end{proof}

We will define a topology on $\Da{\bf A}^{\bowtie}$ such that $\alpha$ is continuous. 
\begin{definition}\label{def:ftopology}
Given an $\sbp$-algebra ${\bf A}$ and clopen up-sets $U\subseteq\Xd{\bool{\bf A}}$, $V\subseteq\Xd{\rad{\bf A}}$, define
$$W_{(U,V)} = [(U\times V)\cup +(U\times\Xd{\delta[\rad{\bf A}]}\cup\Xd{\bool{\bf A}}\times\delta[V]^\comp)]\cap\Da{\bf A}^{\bowtie},$$
where $\delta\colon \rad{\bf A}\to \rad{\bf A}$ is defined by $\delta(x)=\neg\neg x$ as usual, $\delta[V] = \{\delta[\fr{x}] : \fr{x}\in V\}$, and for a subset $P\subseteq\Xd{\bool{\bf A}}\times\Xd{\rad{\bf A}}$, $+P = \{+p : p\in P\}$.
\end{definition}

\begin{remark}\label{rem:extrinsictopology}
Let $\Delta\colon\Xd{\rad{\bf A}}\to\Xd{\rad{\bf A}}$ be defined by $\Delta(\fr{x})=\delta^{-1}[\fr{x}]$ (i.e., $\Delta$ is the dual of the lattice homomorphism given by $\delta(x)=\neg\neg x$). One may easily show that $\Delta$ is a closure operator on $\Xd{\rad{\bf A}}$, and we denote by $\Xd{\rad{\bf A}}_\Delta := \Delta[\Xd{\rad{\bf A}}] = \{\fr{x}\in\Xd{\rad{\bf A}} : \Delta(\fr{x})=\fr{x}\}$ its set of fixed points. Let $\beta\colon\Xd{\rad{\bf A}}_\Delta\to\Xd{\delta[\rad{\bf A}]}$ be the map given by $\beta(\fr{x})=\fr{x}\cap\delta[\rad{\bf A}]$. An identical argument to that given in \cite[Theorem 12 and Lemma 25]{BezGhi} shows that $\beta$ is an isomorphism of Priestley spaces when $\Xd{\rad{\bf A}}_\Delta$ is viewed as a subspace of $\Xd{\rad{\bf A}}$, and the inverse of $\beta$ is given by $\fr{x}\mapsto\Delta(\fr{x})$.

If $V\subseteq\Xd{\rad{\bf A}}$ is a clopen up-set, note also that one may readily show that the image of $\delta[V]=\{\delta[\fr{x}] : \fr{x}\in V\}$ under $\beta^{-1}$ is precisely $\Delta[V] = \{\fr{x}\in V : \Delta(\fr{x})=\fr{x}\} = V\cap \Xd{\rad{\bf A}}_\Delta$. These comments show that we may identify $\Xd{\delta[\rad{\bf A}]}$ and $\Xd{\rad{\bf A}}_\Delta$, as well as $\Delta[V]$ and $\delta[V]$, in the above definition of the topology on $\Da{A}^{\bowtie}$. This provides that the definitions of the sets $W_{(U,V)}$ may be rewritten to depend only on the dual map $\Delta$, and not on $\delta$.

\end{remark}

\begin{lemma}\label{lem:delta image of a clopen}
Let ${\bf A}$ be an $\sbp$-algebra and let $x\in\rad{\bf A}$. Then $\delta[\varphi_{\rad{\bf A}}(x)]=\varphi_{\delta[\rad{\bf A}]}(\delta(x))$.
\end{lemma}

\begin{proof}
Let $\fr{y}\in\delta[\varphi_{\rad{\bf A}}(x)]$. Then there exists $\fr{x}\in\varphi_{\rad{\bf A}}(x)$ with $\delta[\fr{x}]=\fr{y}$. An easy argument (using the fact that $\delta$ is a wdl-admissible map) shows that $\fr{y}=\delta[\fr{x}]\in\Xd{\delta[\rad{\bf A}]}$. Because $x\in\fr{x}$, it follows also that $\delta(x)\in\delta[\fr{x}]=\fr{y}$, whence $\fr{y}\in\varphi_{\delta[\rad{\bf A}]}(\delta(x))$ and thus $\delta[\varphi_{\rad{\bf A}}(x)]\subseteq \varphi_{\delta[\rad{\bf A}]}(\delta(x))$.

For the other inclusion, let $\fr{y}\in \varphi_{\delta[\rad{\bf A}]}(\delta(x))$, and set $\fr{x}=\delta^{-1}[\fr{y}]$. Because $\delta$ is among other things a lattice homomorphism, we have that $\fr{x}\in\Xd{\rad{\bf A}}$. Moreover, $\delta(x)\in\fr{y}$ gives that $x\in\delta^{-1}[\fr{y}]=\fr{x}$, so $\fr{x}\in\varphi_{\rad{\bf A}}(x)$. It is easy to see that $\delta[\fr{x}]=\fr{y}$, and this gives the reverse inclusion and the result.
\end{proof}

We give $\Da{\bf A}^{\bowtie}$ the topology generated by the sets $W_{(U,V)}$ and $W_{(U,V)}^\comp$, where $(U,V)$ ranges over all pairs of clopen up-sets $U\subseteq\Xd{\bool{\bf A}}$ and $V\subseteq\Xd{\rad{\bf A}}$.

\begin{lemma}\label{lem:topology}
When $\Da{\bf A}^{\bowtie}$ is endowed with the topology defined above, $\alpha$ is a continuous map.
\end{lemma}

\begin{proof}
It suffices to show that the inverse image under $\alpha$ of the subbasis elements $W_{(U,V)}$ and $W_{(U,V)}^\comp$ are open. Let $U\subseteq\Xd{\bool{\bf A}}$ and $V\subseteq\Xd{\rad{\bf A}}$ be clopen up-sets. By (extended) Priestley duality, the functions $\varphi_{\bool{\bf A}}\colon \bool{\bf A}\to\Ad{\Xd{\bool{\bf A}}}$ and $\varphi_{\rad{\bf A}}\colon \rad{\bf A}\to\Ad{\Xd{\rad{\bf A}}}$ are isomorphisms. In particular, there hence exist $u\in\bool{\bf A}$ and $x\in\rad{\bf A}$ such that $U=\varphi_{\bool{\bf A}}(u)$ and $V=\varphi_{\rad{\bf A}}(x)$. Setting $a=(u\join\neg x)\meet (\neg u\join x)$, we will show that $\alpha^{-1}[W_{(U,V)}] = \varphi_{\bf A}(a)$.

Let $\fr{a}\in\alpha^{-1}[W_{(U,V)}]$. Then $\alpha(\fr{a})\in W_{(U,V)}$, and we consider two cases. First, if $\fr{a}\subseteq\fr{a}^*$, then $\alpha(\fr{a})=(\fr{a}\cap\bool{\bf A},\fr{a}\cap\rad{\bf A})\in U\times V$, i.e., $\fr{a}\cap\bool{\bf A}\in\varphi_{\bool{\bf A}}(u)$ and $\fr{a}\cap\rad{\bf A}\in\varphi_{\rad{\bf A}}(x)$. Hence $u\in\fr{a}\cap\bool{\bf A}$ and $x\in\fr{a}\cap\rad{\bf A}$, and in particular $u,x\in\fr{a}$. Since $\fr{a}$ is a filter, it follows that $a=(u\join\neg x)\meet (\neg u\join x)\in \fr{a}$, and therefore $\fr{a}\in\varphi_{\bf A}(a)$. Second, if $\fr{a}^*\subset\fr{a}$, then by the definition of $\alpha$ we have that $\alpha(\fr{a})=+(\fr{a}^*\cap\bool{\bf A},\delta[\fr{a}^*\cap\rad{\bf A}])$, where one of $\fr{a}^*\cap\bool{\bf A}\in U=\varphi_{\bool{\bf A}}(u)$ or $\delta[\fr{a}^*\cap\rad{\bf A}]\in \delta[V]^\comp = \delta[\varphi_{\rad{\bf A}}(x)]^\comp$ holds. If $\fr{a}^*\cap\bool{\bf A}\in\varphi_{\bool{\bf A}}(u)$, then $u\in\fr{a}^*$ and hence $u\in\fr{a}$ (since $\fr{a}$ and $\fr{a}^*$ have the same ultrafilter by Lemma \ref{lem:sameultrafiltersite}). In this event, we have that $u\lor\neg x\in\fr{a}$ since $\fr{a}$ is upward-closed. On the other hand, if $\delta[\fr{a}^*\cap\rad{\bf A}]\in \delta[\varphi_{\rad{\bf A}}(x)]^\comp$, then we observe by Lemma \ref{lem:delta image of a clopen} that $\delta[\varphi_{\rad{\bf A}}(x)]^\comp=\varphi_{\delta[\rad{\bf A}]}(\delta(x))^\comp$, and this provides that $\delta(x)\notin\delta[\fr{a}^*\cap\rad{\bf A}]$. This shows in particular that $x\notin\fr{a}^*\cap\rad{\bf A}$, and because $x\in\rad{\bf A}$ we have $x\notin\fr{a}^*$. By the definition of $^*$, this yields $\neg x\in\fr{a}$. Hence $u\join\neg x\in\fr{a}$ once again by $\fr{a}$ being upward-closed. Because $\fr{a}^{*}\subset\fr{a}$ in the present case, Lemma \ref{lemma:incl} gives that $\fr{a}^*\subseteq\scrR_{\fr{u}_\fr{a}}\subseteq\fr{a}$, so in particular $\rad{\bf A}\subseteq\fr{a}$. Hence $x\in\fr{a}$, and this gives $\neg u\join x\in\fr{a}$ as $\fr{a}$ is a filter. It follows that $u\join\neg x,\neg u\join x\in\fr{a}$, so $a=(u\join\neg x)\meet (\neg u\join x)\in\fr{a}$, i.e., $\fr{a}\in\varphi_{\bf A}(a)$. This shows that $\alpha^{-1}[W_{(U,V)}]\subseteq\varphi_{\bf A}(a)$.

We now prove the reverse inclusion, so suppose that $\fr{a}\in\varphi_{\bf A}(a)$. Then $a=(u\join\neg x)\meet (\neg u\join x)\in\fr{a}$, and since $\fr{a}$ is upward-closed we have that $u\join\neg x,\neg u\join x\in\fr{a}$. By primality, we have that both of the following conditions hold: (1) Either $u\in\fr{a}$ or $\neg x\in\fr{a}$, and (2) either $\neg u\in\fr{a}$ or $x\in\fr{a}$. We consider two cases. First, if $\fr{a}\subseteq\fr{a}^*$, then by Lemma \ref{lemma:incl} we have that $\fr{a}\subseteq\scrR_{\fr{u}_{\fr{a}}}\subseteq\fr{a}^*$. 
 $\neg x\notin\fr{a}$, since $\neg x \in \scrC({\bf A})$ and $\fr{a} \subseteq \scrR_{\fr{u}_{\fr{a}}}$, so by condition (1) above $u\in\fr{a}$. Then $\neg u\notin\fr{a}$ because $\fr{a}$ is proper, so by condition (2) above we have $x\in\fr{a}$. Thus we have $u,x\in\fr{a}$, whence $\fr{a}\cap\bool{\bf A}\in\varphi_{\bool{\bf A}}(u)$ and $\fr{a}\cap\rad{\bf A}\in\varphi_{\rad{\bf A}}(x)$, i.e., $\alpha(\fr{a})\in U\times V$. In the second case, we have that $\fr{a}^*\subset\fr{a}$. By the definition of $\alpha$, we then have $\alpha(\fr{a})=+(\fr{a}^*\cap\bool{\bf A},\delta[\fr{a}^*\cap\rad{\bf A}])$. Note that by (1) we have that either $u\in\fr{a}$ or $\neg x\in\fr{a}$. If $u\in\fr{a}$, then $\fr{a}^*\cap\bool{\bf A}=\fr{a}\cap\bool{\bf A}\in\varphi_{\bool{\bf A}}(u)=U$, whence $(\fr{a}^*\cap\bool{\bf A},\delta[\fr{a}^*\cap\rad{\bf A}])\in U\times\Xd{\delta[\rad{\bf A}]}$. If $\neg x\in\fr{a}$, then as $\neg\neg\neg x = \neg x$ we have that $\neg\neg\neg x\in\fr{a}$, so that $\delta(x)=\neg\neg x\notin\fr{a}^*$. This implies that $\delta(x)\notin\fr{a}^*\cap\rad{\bf A}$, and hence $\delta(x)\notin\delta[\fr{a}^*\cap\rad{\bf A}]$, i.e., $\delta[\fr{a}^*\cap\rad{\bf A}]\in\varphi_{\delta[\rad{\bf A}]}(\delta(x))^\comp=\delta[V]^\comp$. This shows that $(\fr{a}^*\cap\bool{\bf A},\delta[\fr{a}^*\cap\rad{\bf A}])\in\Xd{\bool{\bf A}}\times\delta[V]^\comp$. It follows that $\alpha(\fr{a})\in +(U\times\Xd{\delta[\rad{\bf A}]}\cup\Xd{\bool{\bf A}}\times\delta[V]^\comp)$, completing the proof that $\varphi_{\bf A}(a)=\alpha^{-1}[W_{(U,V)}]$.

Because $\alpha^{-1}[W_{(U,V)}^\comp]=(\alpha^{-1}[W_{(U,V)}])^\comp=\varphi_{\bf A}(a)^\comp$ for $a$ as above, this shows that the $\alpha$-inverse image of subbasis elements are open (indeed, subbasis elements). This proves that $\alpha$ is continuous.
\end{proof}

Throughout the remainder of this investigation, we assume without further mention that $\Da{\bf A}^{\bowtie}$ is equipped with the topology generated by the sets $W_{(U,V)},W_{(U,V)}^\comp$. Note that the above actually shows more. Because clopen subbasis elements of $\Xd{\bf A}$ and $\Da{\bf A}^{\bowtie}$ precisely correspond under the order isomorphism $\alpha$, all structure is transported from $\Xd{\bf A}$ to $\Da{\bf A}^{\bowtie}$. In particular, $\Da{\bf A}^{\bowtie}$ is a Priestley space that is isomorphic as a Priestley space under $\alpha$ to $\Xd{\bf A}$.

\begin{example}
In order to build intuition, we will now give an example of the construction. 
Consider the Chang MV-algebra ${\bf C}$, with domain $C = \{0,c,\ldots, nc, \ldots, 1-nc, \ldots,1-c,1\}$. 
Let  $C^{+} = \{1, 1-c, \ldots, 1-nc,\ldots \}$ and  $C^{-} = \{0,c,\ldots,nc,\ldots\}$. It is easy to see that $C^{+}$ is isomorphic to the cancellative hoop given by the negative cone of the integers $(\textsf{Z}^{-}, +, \ominus, {\rm min}, {\rm max}, 0)$, where $\ominus$ is the difference truncated to $0$.
${\bf C}$ is a perfect MV-algebra, and is generic for the variety $\textsf{DLMV}$. Let us consider the DLMV-algebra ${\bf C}^{2} = {\bf C} \times {\bf C}$, as in Figure \ref{fig:cxc}. Notice that the Boolean skeleton of ${\bf C}^{2}$ is the four-element Boolean algebra, $\scrB({\bf C}^{2}) = \{(0,0), (0,1), (1,0), (1,1)\}$, while the radical $\scrR({\bf C}^{2})$ is isomorphic to $\textsf{Z}^{-} \times \textsf{Z}^{-}$, and is the upper square of Figure \ref{fig:cxc}.\vspace{0.2cm}\\
 \begin{figure}[!h]
 \begin{center}
\begin{tikzpicture}
     \fill[gray!20!white](0,4) -- (-0.9,3.1) -- (0, 2.2) -- (0.9,3.1);
      \fill[gray!20!white](0,0) -- (-0.9,0.9) -- (0,1.8) -- (0.9,0.9);
       \fill[gray!20!white](-2,2) -- (-1.1,2.9) -- (-0.2,2) -- (-1.1,1.1);
        \fill[gray!20!white](2,2) -- (1.1,1.1) -- (0.2,2) -- (1.1,2.9);
        
         \node at (0,0) [draw, circle, fill=black, minimum size=2pt, inner sep=0pt, label distance=1mm]{};
 
  \node at (0,4) [draw, circle, fill=black, minimum size=2pt, inner sep=0pt, label distance=1mm]{};
   \node at (-2,2) [draw, circle, fill=black, minimum size=2pt, inner sep=0pt, label distance=1mm]{};
    \node at (2,2) [draw, circle, fill=black, minimum size=2pt, inner sep=0pt, label distance=1mm]{};
    \node at (1.5,1.5) [draw, circle, fill=black, minimum size=2pt, inner sep=0pt, label distance=1mm]{};
      \node at (0.5,0.5) [draw, circle, fill=black, minimum size=2pt, inner sep=0pt, label distance=1mm]{};
        \node at (-0.3,0.3) [draw, circle, fill=black, minimum size=2pt, inner sep=0pt, label distance=1mm]{};
      \node at (0.4,2.6) [draw, circle, fill=black, minimum size=2pt, inner sep=0pt, label distance=1mm]{};
       \node at (-1.7,1.7) [draw, circle, fill=black, minimum size=2pt, inner sep=0pt, label distance=1mm]{};
      \node at (0.4,2.6) [draw, circle, fill=black, minimum size=2pt, inner sep=0pt, label distance=1mm]{};
       \node at (-0.6,2.8) [draw, circle, fill=black, minimum size=2pt, inner sep=0pt, label distance=1mm]{};
        
    \draw (-0.5,3.5) [line width=0.5pt, dashed]  to (1.5,1.5);
      \draw (0.3,3.7) [line width=0.5pt, dashed]  to (-1.7,1.7);
          \draw (-0.3,0.3) [line width=0.5pt, dashed]  to (1.7,2.3);
      \draw (0.5,0.5) [line width=0.5pt, dashed]  to (-1.5,2.5);
             
    \draw (0,4) [line width=0.5pt,]  to (-0.9,3.1);
    \draw (-2,2) [line width=0.5pt,]  to (-1.1,2.9);
      \draw (0,0) [line width=0.5pt,]  to (-0.9,0.9);
    \draw (-2,2) [line width=0.5pt,]  to (-1.1,1.1);
      \draw (0,0) [line width=0.5pt,]  to (0.9,0.9);
    \draw (2,2) [line width=0.5pt,]  to (1.1,1.1);
      \draw (0,4) [line width=0.5pt,]  to (0.9,3.1);
    \draw (2,2) [line width=0.5pt,]  to (1.1,2.9);
    \small{
    \node at (0,4.4) {$(1,1)$};
    \node at (-2.6,2) {$(1,0)$};
    \node at (2.6,2) {$(0,1)$};
     \node at (0,-0.4) {$(0,0)$};
 \node at (2.3,1.3) {$(0, 1- nc)$};   
 \node at (1.05,0.3) {$(0, nc)$};   
  \node at (-2.5,1.45) {$(1- mc, 0)$};   
    \node at (-1,0.2) {$(mc, 0)$};     
 \footnotesize{ \node at (0.65,2.6) {$r_{n}$};  
 \node at (-0.35,2.65) {$r_{m}$};  
 }
     }
 \end{tikzpicture}
 \end{center}
  \caption{${\bf C}^{2}$.}\label{fig:cxc}
 \end{figure}
We shall now construct $\mathcal{F}_{ \bf C^{2}}$. Call $\mathfrak{u}_{1}$ the Boolean ultrafilter generated by $(1,0)$, $\mathfrak{u}_{2}$ the Boolean ultrafilter generated by $(0,1)$, $C^{+}_{1}$ the prime filter of the radical given by the segment $\{(1, y) : y \in C^{+}\}$, and $C^{+}_{2}$ the one given by the segment $\{(x, 1) : x \in C^{+}\}$. Looking at Figure \ref{fig:chdual}, it is easy to see that the pairs in $\mathcal{F}_{ C^{2}}$ will be of the kind $(\mathfrak{u}_{2}, [r_{n}))$, $(\mathfrak{u}_{1}, [r_{m}))$, plus the pairs $(\mathfrak{u}_{1}, \scrR({\bf C}^{2}))$, $(\mathfrak{u}_{1}, C^{+}_{1}))$ and $(\mathfrak{u}_{2}, \scrR({\bf C}^{2}))$, $(\mathfrak{u}_{2}, C^{+}_{2})$.
In particular, via the isomorphism $\alpha$ of Theorem \ref{th:dualiso}, we have the following correspondences: 
\begin{itemize}
\item $(\mathfrak{u}_{2}, [r_{n}))$ will correspond to the prime filter of ${\bf C}^{2}$ generated by the element $(0, 1- nc)$;
\item $(\mathfrak{u}_{1}, [r_{m}))$ to the prime filter generated by $(1- mc, 0)$;
\item $(\mathfrak{u}_{1}, \scrR({\bf C}^{2}))$ to the prime filter given by the upper and left squares of Figure \ref{fig:cxc}; analogously, $(\mathfrak{u}_{2}, \scrR({\bf C}^{2}))$ corresponds to upper-right squares;
 \item $(\mathfrak{u}_{1}, C^{+}_{1})$ to the segment $\{(1,y) : y \in C\}$, and $(\mathfrak{u}_{2}, C^{+}_{2})$ to $\{(x, 1) : x \in C\}$.
\end{itemize}

Now, we want to construct $\mathcal{F}_{\bf C^{2}}^{\bowtie}$. In ${\bf C}^{2}$ the double negation, and hence the $\delta$ of the construction, is the identity map. Thus, intuitively, we just need to rotate upwards $\mathcal{F}_{\bf C^{2}}$ to obtain $\mathcal{F}_{\bf C^{2}}^{\partial}$, and then $\mathcal{F}_{\bf C^{2}}^{\bowtie}$ is as in Figure \ref{fig:chdual}. 
Again via the isomorphism we have the following correspondences:
\begin{itemize}
\item $(\mathfrak{u}_{1}, \delta[[r_{m})])$ will correspond to the prime $\ell$-filter of ${\bf C}^{2}$ generated by the element $(mc, 0)$;
\item $(\mathfrak{u}_{2}, \delta [[r_{n})])$ to the prime $\ell$-filter generated by $(0, nc)$;
 \item $(\mathfrak{u}_{1}, \delta[C^{+}_{1}])$ to the prime $\ell$-filter given by ${\bf C}^{2} \setminus \{(0,y) : y \in C\}$, and $(\mathfrak{u}_{2}, \delta[C^{+}_{2}])$ to ${\bf C}^{2} \setminus \{(x, 0) : x \in C\}$.
\end{itemize}
It is now easy to realize that $\mathcal{F}_{ C^{2}}^{\bowtie}$ is order isomorphic to the poset of prime filters of ${\bf C}^{2}$.

 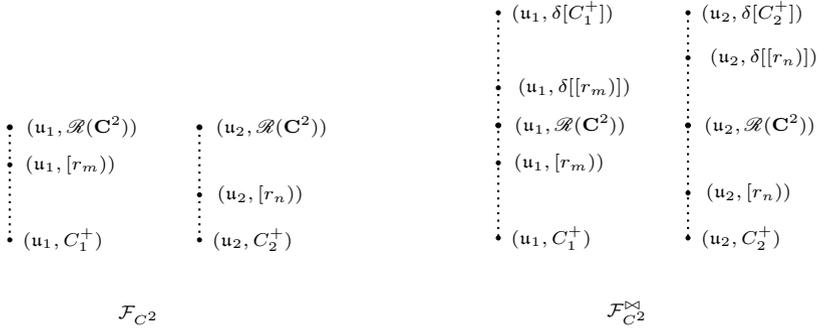
\begin{figure}
  \begin{center}
\begin{tikzpicture}
    
         \node at (0,0) [draw, circle, fill=black, minimum size=1.8pt, inner sep=0pt, label distance=1mm]{};
           \node at (0,-1.5) [draw, circle, fill=black, minimum size=1.5pt, inner sep=0pt, label distance=1mm]{};
    \node at (2.5,0) [draw, circle, fill=black, minimum size=1.8pt, inner sep=0pt, label distance=1mm]{};
       \node at (2.5,-0.9) [draw, circle, fill=black, minimum size=1.5pt, inner sep=0pt, label distance=1mm]{};
       \node at (2.5,-1.5) [draw, circle, fill=black, minimum size=1.5pt, inner sep=0pt, label distance=1mm]{};
           \node at (0,-0.5) [draw, circle, fill=black, minimum size=1.5pt, inner sep=0pt, label distance=1mm]{};
              
    \draw (0,0) [line width=0.8pt, dotted]  to (0,-1.5);  
    \draw (2.5,0) [line width=0.8pt, dotted]  to (2.5,-1.5);

    \scriptsize{
    \node at (0.95,0) {$(\mathfrak{u}_{1},\scrR({\bf C}^{2}))$};  
    \node at (0.7,-1.5) {$(\mathfrak{u}_{1},C^{+}_{1})$};  
     \node at (3.45,0) {$(\mathfrak{u}_{2}, \scrR({\bf C}^{2}))$}; 
     \node at (3.3,-0.9) {$(\mathfrak{u}_{2},[r_{n}))$};  
         \node at (0.8,-0.5) {$(\mathfrak{u}_{1},[r_{m}))$};  
      \node at (3.2,-1.5) {$(\mathfrak{u}_{2},C^{+}_{2})$};  
             \node at (1.7,-2.5) {$\mathcal{F}_{ C^{2}}$}; 
     }
 \end{tikzpicture}\hspace{2cm}
\begin{tikzpicture}
    
        \node at (0,0) [draw, circle, fill=black, minimum size=1.8pt, inner sep=0pt, label distance=1mm]{};
           \node at (0,-1.5) [draw, circle, fill=black, minimum size=1.5pt, inner sep=0pt, label distance=1mm]{};
             \node at (0,1.5) [draw, circle, fill=black, minimum size=1.5pt, inner sep=0pt, label distance=1mm]{};
    \node at (2.5,0) [draw, circle, fill=black, minimum size=1.8pt, inner sep=0pt, label distance=1mm]{};
       \node at (2.5,-0.9) [draw, circle, fill=black, minimum size=1.5pt, inner sep=0pt, label distance=1mm]{};
       \node at (2.5,0.9) [draw, circle, fill=black, minimum size=1.5pt, inner sep=0pt, label distance=1mm]{};
        \node at (0,-0.5) [draw, circle, fill=black, minimum size=1.5pt, inner sep=0pt, label distance=1mm]{};
         \node at (0,0.5) [draw, circle, fill=black, minimum size=1.5pt, inner sep=0pt, label distance=1mm]{};
       \node at (2.5,-1.5) [draw, circle, fill=black, minimum size=1.5pt, inner sep=0pt, label distance=1mm]{};
       \node at (2.5,1.5) [draw, circle, fill=black, minimum size=1.5pt, inner sep=0pt, label distance=1mm]{};
                   
    \draw (0,1.5) [line width=0.8pt, dotted]  to (0,-1.5);  
    \draw (2.5,1.5) [line width=0.8pt, dotted]  to (2.5,-1.5);

    \scriptsize{
    \node at (0.95,0) {$(\mathfrak{u}_{1},\scrR({\bf C}^{2}))$};  
    \node at (0.7,-1.5) {$(\mathfrak{u}_{1},C^{+}_{1})$};  
     \node at (3.45,0) {$(\mathfrak{u}_{2}, \scrR({\bf C}^{2}))$}; 
     \node at (0.85,1.5) {$(\mathfrak{u}_{1},\delta[C^{+}_{1}])$};  
      \node at (0.8,-0.5) {$(\mathfrak{u}_{1},[r_{m}))$};  
       \node at (1,0.5) {$(\mathfrak{u}_{1},\delta[[r_{m})])$};  
     \node at (3.3,-0.9) {$(\mathfrak{u}_{2},[r_{n}))$};  
       \node at (3.5,0.9) {$(\mathfrak{u}_{2},\delta[[r_{n})])$};  
      \node at (3.2,-1.5) {$(\mathfrak{u}_{2},C^{+}_{2})$};  
      \node at (3.35,1.5) {$(\mathfrak{u}_{2},\delta[C^{+}_{2}])$};  
        \node at (1.7,-2.5) {$\mathcal{F}_{ C^{2}}^{\bowtie}$}; 
     }
 \end{tikzpicture}
 \end{center}
 \caption{$\mathcal{F}_{ C^{2}}$ and $\mathcal{F}_{C^{2}}^{\bowtie}$ }\label{fig:chdual}
 \end{figure}
 \begin{remark}
(1) As from Lemma \ref{lemmafixes}, pairs $(\mathfrak{u}, \mathfrak{x}) \in \mathcal{F}_{\bf C^{2}}$ are such that $\mathfrak{u}$ fixes $\mathfrak{x}$. In particular, this means that for every $b \notin \mathfrak{u}$, $\mathfrak{x}$ is a fix point of $\upup_{b}$. For example, consider a pair of the kind $(\mathfrak{u}_{2},[r_{n}))$. We have that $\upup_{(1,0)}([r_{n})) = [r_{n})$, since the elements of the filter generated by $r_{n}$ are the only elements of $\scrR({\bf C}^{2})$ whose join with $(1,0)$ is in $[r_{n})$.
 
(2) It is worth noticing that the only implicative prime filters are the ones given by, respectively, upper-left and upper-right squares of Figure \ref{fig:cxc}, and the two segments $\{(1,y) : y \in C\}$ and $\{(x, 1) : x \in C\}$. They correspond respectively to the pairs $(\mathfrak{u}_{1},\scrR({\bf C}^{2}))$, $(\mathfrak{u}_{2},\scrR({\bf C}^{2}))$, $(\mathfrak{u}_{1},C^{+}_{1})$ and $(\mathfrak{u}_{2},C^{+}_{2})$.
 \end{remark}
\end{example}

\section{Multiplying filters}\label{sec:multiplyingfilters}

The map $\alpha$ of the Section \ref{sec:filterpairs} gives a Priestley isomorphism between $\Xd{\bf A}$ and $\Da{A}^{\bowtie}$ for any $\sbp$-algebra ${\bf A}$. Next we will define an appropriate ternary relation on $\Da{A}^{\bowtie}$ under which $\alpha$ will be an isomorphism in $\mtl^\tau$. For this, our presentation of dual relations in terms of partial binary operations in Section \ref{sec:duality} will prove useful. Given an $\sbp$-algebra ${\bf A}$ and $\fr{a},\fr{b}\in\Xd{A}$, the ultrafilters of $\fr{a}$ and $\fr{b}$ have a profound impact on $\fr{a}\bullet\fr{b}$.

\begin{lemma}\label{lem:filtmulttriv}
Let ${\bf A}$ be an $\sbp$-algebra, and let $\fr{a},\fr{b}\in\Xd{\bf A}$. Then $\fr{u}_\fr{a}\neq\fr{u}_\fr{b}$ implies that $\fr{a}\bullet\fr{b}=A$.
\end{lemma}

\begin{proof}
Since $\fr{u}_\fr{a}\neq\fr{u}_\fr{b}$, we may assume without loss of generality that there exists $u\in\fr{u}_\fr{a}\subseteq\fr{a}$ with $u\notin\fr{u}_\fr{b}$. Because $\fr{u}_\fr{b}$ is an ultrafilter of $\bool{\bf A}$, $u\notin\fr{u}_\fr{b}$ implies that $\neg u\in\fr{u}_\fr{b}\subseteq\fr{b}$. It follows that $u\cdot\neg u = u\meet\neg u = 0\in\fr{a}\cdot\fr{b}$. Since $\fr{a}\bullet\fr{b}$ is upward-closed, this yields $\fr{a}\bullet\fr{b}=A$.
\end{proof}

As a consequence of Lemma \ref{lem:filtmulttriv}, understanding the multiplication of prime filters on an $\sbp$-algebras amounts to understanding the products of prime filters with a given ultrafilter. Recall that $\site{u} =  \{\fr{a}\in\Xd{A} : \fr{u}\subseteq\fr{a}\}$. 
%
%

\begin{lemma}\label{lem:doubleneg}
Let ${\bf A}$ be an $\sbp$-algebra and let $\fr{a}\in\Xd{\bf A}$ with $\rad{\bf A}\subseteq\fr{a}$. Then $\fr{a}=\fr{a}^{**}$.
\end{lemma}

\begin{proof}
Observe that $a\in\fr{a}^{**}$ if and only if $\neg\neg a\in\fr{a}$ by the definition of $^*$, so it suffices to show that $a\in\fr{a}$ if and only if $\neg\neg a\in\fr{a}$. The fact that $a\leq\neg\neg a$ gives that $a\in\fr{a}$ implies $\neg\neg a\in\fr{a}$ by $\fr{a}$ being upward-closed, so suppose that $\neg\neg a\in\fr{a}$. Write $a=(u\meet\neg x)\join (\neg u\meet x)$, where $u\in\bool{\bf A}$ and $x\in\rad{\bf A}$. Then $\neg\neg a = (u\meet\neg\neg\neg x)\join (\neg u\meet \neg \neg x)= (u\meet \neg x)\join (\neg u\meet\neg\neg x)\in\fr{a}$. By primality it follows that $u\meet \neg x\in\fr{a}$ or $\neg u\meet\neg\neg x\in\fr{a}$. In the former case, $a\in\fr{a}$ follows from $u\meet\neg x\leq a$. In the latter case, $\neg u\meet\neg\neg x\in\fr{a}\leq \neg u$ gives that $\neg u\in\fr{a}$, and since $x\in\rad{\bf A}\subseteq\fr{a}$ we obtain that $\neg u\meet x\in\fr{a}$. Thus by $\neg u\meet x\leq a$ and $\fr{a}$ being upward-closed we get $a\in\fr{a}$, concluding the proof.
\end{proof}

\begin{lemma}\label{lem:extrinsic mult}
Let ${\bf A}$ be an $\sbp$-algebra, let $\fr{u}$ be an ultrafilter of $\bool{\bf A}$, and let $\fr{a},\fr{b}\in\site{u}$. Denote by $\prodrad$ and $\resrad$ the operations on $\Xd{\rad{\bf A}}$ defined as in Section \ref{sec:duality}. Then we have the following.

\begin{enumerate}
\item If $\fr{a},\fr{b}\subseteq\scrR_\fr{u}$, then $\fr{a}\bullet\fr{b}= \langle\fr{u}\cup (\fr{a}\cap\rad{\bf A})\prodrad (\fr{b}\cap\rad{\bf A})\rangle$.
\item If $\fr{a}\subseteq\fr{b}^*\subseteq\scrR_\fr{u}\subseteq\fr{b}$, then $\fr{a}\bullet\fr{b}=\langle \fr{u}\cup ((\fr{a}\cap\rad{\bf A})\resrad (\fr{b}^*\cap\rad{\bf A}))\rangle^*$.
\item If none of $\fr{a},\fr{b}\subseteq\scrR_\fr{u}$, $\fr{a}\subseteq\fr{b}^*\subseteq\scrR_\fr{u}\subseteq\fr{b}$, or $\fr{b}\subseteq\fr{a}^*\subseteq\scrR_\fr{u}\subseteq\fr{a}$ hold, then $\fr{a}\bullet\fr{b}=A$.
\end{enumerate}
\end{lemma}

\begin{proof}
For (1), observe that $\fr{u}$ is the ultrafilter of $\fr{a}\bullet\fr{b}$, and that $\fr{a}\bullet\fr{b}\subseteq\scrR_\fr{u}$ by the fact that $\bullet$ is order-preserving and $\scrR_\fr{u}\bullet\scrR_\fr{u}=\scrR_\fr{u}$. It thus suffices to show that $\fr{a}\bullet\fr{b}\cap\rad{\bf A}=(\fr{a}\cap\rad{\bf A})\prodrad (\fr{b}\cap\rad{\bf A})$. Let $c\in\fr{a}\bullet\fr{b}\cap\rad{\bf A}$. Then $c\in\rad{\bf A}$, and there exist $a\in\fr{a}$ and $b\in\fr{b}$ such that $a\cdot b\leq c$. Then $a,c\leq a\join c$ and $b,c\leq b\join c$ and the fact that $\fr{a},\fr{b},\rad{\bf A}$ are upward-closed gives that $a\join c\in\fr{a}\cap\rad{\bf A}$ and $b\join c\in\fr{b}\cap\rad{\bf A}$. Observe that $(a\join c)\cdot (b\join c)=ab\join ac\join bc\join c^2\leq c$, whence $c\in (\fr{a}\cap\rad{\bf A})\prodrad (\fr{b}\cap\rad{\bf A})$ and hence $\fr{a}\bullet\fr{b}\cap\rad{\bf A}\subseteq (\fr{a}\cap\rad{\bf A})\prodrad (\fr{b}\cap\rad{\bf A})$.

For the reverse inclusion, let $c\in (\fr{a}\cap\rad{\bf A})\prodrad (\fr{b}\cap\rad{\bf A})$. Then there exist $a\in\fr{a}\cap\rad{\bf A}$, $b\in\fr{b}\cap\rad{\bf A}$, with $a\cdot b\leq c$. Because the radical is closed under $\cdot$, we have that $a\cdot b\in\rad{\bf A}$, so $c\in\rad{\bf A}$ as well. It follows that $a\in\fr{a}$, $b\in\fr{b}$, and $c\in\rad{\bf A}$, and hence $c\in\fr{a}\bullet\fr{b}\cap\rad{\bf A}$. Equality follows.

We now prove (2). Let $\fr{a} = \langle \fr{u} \cup \fr{y} \rangle$ and $\fr{b}^* = \langle \fr{u} \cup \fr{x} \rangle$, so that $\fr{b}=\fr{b}^{**}=\langle\fr{u}\cup\fr{x}\rangle^*$ by Lemma \ref{lem:doubleneg}. Observe that since $\fr{a}\subseteq\fr{b}^*$, we have also that $\fr{a}\cap\rad{\bf A}\subseteq\fr{b}^*\cap\rad{\bf A}$. This implies that $\{1\}\prodrad(\fr{a}\cap\rad{\bf A})\subseteq\fr{b}^*\cap\rad{\bf A}$, and by Lemma \ref{lem:filter to prime filter}(2) there exists a prime filter $\fr{z}\in\Xd{\rad{\bf A}}$ such that $\fr{z}\prodrad (\fr{a}\cap\rad{\bf A})\subseteq\fr{b}^*\cap\rad{\bf A}$. This shows that $\fr{y}\resrad\fr{x}=(\fr{a}\cap\rad{\bf A})\resrad (\fr{b}^*\cap\rad{\bf A})\neq\emptyset$ is a generalized prime filter of $\rad{\bf A}$, so that $\resrad$ is defined in this instance. We prove that $$\langle \fr{u} \cup \fr{y} \rangle \bullet  \langle \fr{u} \cup \fr{x} \rangle^{*} = \langle \fr{u} \cup (\fr{y} \resrad \fr{x})\rangle^{*}.$$
If $a \in \langle \fr{u} \cup \fr{y} \rangle \bullet  \langle \fr{u} \cup \fr{x} \rangle^{*}$, then there exist $w \in \langle \fr{u} \cup \fr{y} \rangle, z \in  \langle \fr{u} \cup \fr{x} \rangle^{*}$, such that $z \cdot w \leq a$. This via Corollary \ref{cor:Rustar} means that there exist $b, b' \in \fr{u}, \neg\neg c \in \scrR({\bf A}) \setminus \fr{x}, d \in \fr{y}$ such that $b \land \neg c \leq z$, $b' \land d \leq w$, thus $(b \land \neg c) \cdot (b' \land d) \leq z \cdot w \leq a$. It is easy to see using directly indecomposable components that $(b \land \neg c) \cdot (b' \land d) = (b \cdot b') \land (\neg c \cdot d)$. Notice that $\neg c \cdot d \in \scrC({\bf A})$, thus there exists $z \in \scrR({\bf A})$ such that $\neg z = \neg c \cdot d$.  We want to prove that $z \notin \delta[\fr{y} \resrad \fr{x}]$, because this will imply that $a \in  \langle \fr{u} \cup (\fr{y} \resrad \fr{x})\rangle^{*}$ by Lemma \ref{cor:Rustar}. 
Let us suppose by contradiction that $z \in \fr{y} \resrad \fr{x}$. Then $z \cdot y \in \fr{x}$ for every $y \in \fr{y}$ and in particular $z \cdot d \in \fr{x}$. Now, $z \cdot \neg z = z \cdot (\neg c \cdot d) = 0$ which implies that $z \cdot d \leq \neg\neg c$ thus $\neg\neg c \in \fr{x}$, a contradiction.

We now need to prove that $\langle \fr{u} \cup (\fr{y} \resrad \fr{x})\rangle^{*} \subseteq \langle \fr{u} \cup \fr{y} \rangle \bullet  \langle \fr{u} \cup \fr{x} \rangle^{*}$, so let $a\in \langle \fr{u} \cup (\fr{y} \resrad \fr{x})\rangle^{*}$. Then by Lemma \ref{cor:Rustar} there exists $u\in\fr{u}$, $\neg\neg z\notin\delta[\fr{y}\resrad\fr{x}]$ such that $u\meet\neg z\leq a$. As $\neg\neg z\notin\delta[\fr{y}\resrad\fr{x}]$, it follows that $z\notin \fr{y}\resrad\fr{x}$, whence there exists $y\in\fr{y}$ such that $yz\notin\fr{x}$. This gives that $yz\notin\langle\fr{u}\cup\fr{x}\rangle$, and since $\fr{b}^*=\langle\fr{u}\cup\fr{x}\rangle$ we have that $\neg\neg (yz)\notin \langle\fr{u}\cup\fr{x}\rangle$. The latter follows because, by the definition of $^*$,
$$\neg\neg x\in\fr{b}^*\iff \neg\neg\neg x\notin\fr{b}\iff \neg x\notin\fr{b}\iff x\in\fr{b}^*.$$
Because $\neg\neg (yz)\notin\langle\fr{u}\cup\fr{x}\rangle$, it follows that $\neg (yz)\in\langle\fr{u}\cup\fr{x}\rangle^*$. Observe:
\begin{align*}
\neg (yz) &= (yz)\to 0\\
&= y\to (z\to 0)\\
&= y\to\neg z.
\end{align*}
Thus $y\to\neg z\in\langle\fr{u}\cup\fr{x}\rangle^*$, and we get $\neg z\in\langle\fr{u}\cup\fr{y}\rangle\bullet\langle\fr{u}\cup\fr{x}\rangle^*$ as $y(y\to\neg z)\leq\neg z$. Since $u\in\langle\fr{u}\cup\fr{y}\rangle\bullet\langle\fr{u}\cup\fr{x}\rangle^*$, we obtain that $u\meet\neg z\in\langle\fr{u}\cup\fr{y}\rangle\bullet\langle\fr{u}\cup\fr{x}\rangle^*$, from which we get that $a$ is contained in the latter set as $u\meet\neg z\leq a$. This gives the reverse inclusion, yielding equality and (2).

For (3), note that the hypothesis guarantees that $\fr{a}\not\subseteq\fr{b}^*$ and $\fr{b}\not\subseteq\fr{a}^*$. Because $\fr{c}^*$ is the largest element of $\Xd{\bf A}$ such that $\fr{c}\bullet\fr{c}^*\neq A$ by Lemma \ref{lem:routley star definition 1}, it follows that $\fr{a}\bullet\fr{b}=A$.
\end{proof}

Given an $\sbp$-algebra ${\bf A}$, Lemma \ref{lem:extrinsic mult} provides a complete description of the partial operation $\bullet$ on $\Xd{\bf A}$ in terms of operation the $\prodrad$ and partial operation $\resrad$ on $\Xd{\rad{\bf A}}$. Specifically, the following rephrases Lemma \ref{lem:extrinsic mult} in terms of $\Da{\bf A}^{\bowtie}$ by employing Proposition \ref{prop:prime} and the isomorphism $\alpha$.

\begin{corollary}\label{cor:extrinsic mult with da}
Let ${\bf A}$ be an $\sbp$-algebra and let $\fr{a},\fr{b}\in\site{u}$ for some $\fr{u}\in\Xd{\bool{\bf A}}$. Then we have the following.
\begin{enumerate}
\item If $\alpha(\fr{a})=(\fr{u},\fr{x})$, and $\alpha(\fr{b})=(\fr{u},\fr{y})$ are in $\Da{A}$, then we have $\alpha(\fr{a}\bullet\fr{b})=(\fr{u},\fr{x}\prodrad\fr{y})$.
\item If $\alpha(\fr{a})=(\fr{u},\fr{x})\in\Da{A}$ and $\alpha(\fr{b})=+(\fr{u},\fr{y})\in\Da{A}^\partial$ with $(\fr{u},\fr{x})\sqsubseteq (\fr{u},\delta^{-1}[\fr{y}])$, then $\alpha(\fr{a}\bullet\fr{b})=+(\fr{u},\fr{x}\resrad\delta^{-1}[\fr{y}])$.
\end{enumerate}
\end{corollary}

Owing to the above, for any $\sbp$-algebra ${\bf A}$ we may define a partial operation $\circ$ on $\Da{A}^{\bowtie}$ by
\begin{enumerate}
\item $(\fr{u},\fr{x})\circ (\fr{u},\fr{y})=(\fr{u},\fr{x}\prodrad\fr{y})$ for any $(\fr{u},\fr{x}),(\fr{u},\fr{y})\in\Da{A}$.
\item $(\fr{u},\fr{x})\circ +(\fr{u},\fr{y})=+(\fr{u},\fr{x}\resrad\delta^{-1}[\fr{y}])$ for any $(\fr{u},\fr{x})\in\Da{A}$, $+(\fr{u},\fr{y})\in\Da{A}^\partial$ with $(\fr{u},\fr{x})\sqsubseteq (\fr{u},\delta^{-1}[\fr{y}])$.
\item $+(\fr{u},\fr{y})\circ (\fr{u},\fr{x}) = +(\fr{u},\fr{x}\resrad\delta^{-1}[\fr{y}])$ for any $(\fr{u},\fr{x})\in\Da{A}$, $+(\fr{u},\fr{y})\in\Da{A}^\partial$ with $(\fr{u},\fr{x})\sqsubseteq (\fr{u},\delta^{-1}[\fr{y}])$.
\item $\circ$ undefined otherwise.
\end{enumerate}
With the above definition, for $\fr{a},\fr{b}\in\Xd{A}$ we have that $\fr{a}\bullet\fr{b}$ is defined if and only if $\alpha(\fr{a})\circ\alpha(\fr{b})$ is defined, and in this case. $\alpha(\fr{a}\bullet\fr{b})=\alpha(\fr{a})\circ\alpha(\fr{b})$. The Priestley isomorphism $\alpha$ is then an isomorphism in $\mtl^\tau$ with respect to ternary relations corresponding to $\bullet$ and $\circ$ as in Section \ref{sec:duality}.

\section{Dual quadruples and the dual construction}\label{sec:dualquadruples}
The stage is finally set to describe the dual of the construction of \cite{AguzFlamUgol}.
\begin{definition}\label{def:dualquad}
A \emph{dual quadruple} is a structure $({\bf S}, {\bf X}, \Upupsilon, \Delta)$ where
\begin{enumerate}
\item ${\bf S}$ is a Stone space;
\item ${\bf X}$ is in {\sf GMTL$^{\tau}$};
\item\label{def:dualquadjoin} $\Upupsilon= \{\upup_{\sf U}\}_{{\sf U} \in \mathcal{A}({\bf S})}$ is an indexed family of $\gmtl^\tau$-morphisms $\upup_{\sf U}: {\bf X} \to {\bf X}$ such that the map $\join_e\colon\Ad{\bf S}\times\Ad{\bf X}\to\Ad{\bf X}$ defined by
$$\join_e({\sf U}, {\sf X})=\upup_{\sf U}^{-1}[{\sf X}]$$
is an external join;
\item\label{cond:delta} $\Delta: {\bf X} \to {\bf X}$ is a continuous closure operator such that $R({\sf x}, {\sf y}, {\sf z})$ implies $R(\Delta {\sf x}, \Delta {\sf y}, \Delta {\sf z})$.
\end{enumerate}
\end{definition}

\begin{definition}
Let $({\bf S}, {\bf X}, \Upupsilon, \Delta)$ be a dual quadruple. We say that ${\sf{u}} \in S$ fixes ${\sf{x}} \in X$ if for every $\sf{U} \subseteq S$ clopen with $\sf{u} \notin \sf{U}$, $\upup_{\sf u}({\sf x}) = {\sf x}$.
\end{definition}

\begin{definition}\label{def:dual}Given a dual quadruple $({\bf S}, {\bf X}, \Upupsilon, \Delta)$, we construct an extended Priestley space $\SnuX$ as follows. The carrier has two parts. $ D = \{(\sf{u},\sf{x}) : u \mbox{ fixes } x \}$ with product order, and $\mathcal{D} ^{\partial} = \{+({\sf u}, \Delta(\sf{x})) : (\sf{u}, \sf{x}) \in \mathcal{D}, x \neq \top\}$, with reverse order. Set $T = D \cup D^{\partial}$. We define a partial order $\sqsubseteq$ on $T$ by $\sf{p}\sqsubseteq \sf{q}$ if and only if
\begin{enumerate}
\item $\sf{p}=(\sf{u},\sf{x})$ and $\sf{q}=(\sf{v},\sf{y})$ for some $(\sf{u},\sf{x}),(\sf{v},\sf{y})\in D$ with $\sf{u}=\sf{v}$ and $\sf{x}\leq\sf{y}$,
\item $\sf{p}=+(\sf{u}',\sf{x}')$ and $\sf{q}=+(\sf{v}',\sf{y}')$ for some $+(\sf{u}',\sf{x}'),+(\sf{v}',\sf{y}')\in D^\partial$ with $\sf{u}'=\sf{v}'$ and $\sf{y}'\leq\sf{x}'$, or
\item $\sf{p}=(\sf{u},\sf{x})$ and $\sf{q}=+(\sf{v},\sf{y})$ for some $(\sf{u},\sf{x})\in D$, $(\sf{v},\sf{y})\in D^\partial$ with $\sf{u}=\sf{v}$.
\end{enumerate}
For each ${\sf U}\in\Ad{\bf S}$, ${\sf V}\in\Ad{\bf X}$, define
$${\sf W}_{({\sf U},{\sf V})} = [({\sf U}\times {\sf V})\cup +({\sf U}\times \Delta[X]\cup S\times\Delta[V]^\comp)]\cap T,$$
and endow $\SnuX$ with the topology generated by the subbase consisting of the sets ${\sf W}_{({\sf U},{\sf V})}$ and ${\sf W}_{({\sf U},{\sf V})}^\comp$.
Further, define a partial binary operation $\circ$ on $\SnuX$ by the following, where $\bullet$ and $\res$ denote the partial operations on ${\bf X}$ given as in Section \ref{sec:functional}.
\begin{enumerate}
\item $(\sf{u},\sf{x})\circ (\sf{u},\sf{y})=(\sf{u},\sf{x}\bullet\sf{y})$ for any $(\sf{u},\sf{x}),(\sf{u},\sf{y})\in\Da{A}$.
\item $(\sf{u},\sf{x})\circ +(\sf{u},\sf{y})=+(\sf{u},\sf{x}\res\Delta(\sf{y}))$ for any $(\sf{u},\sf{x})\in\Da{A}$, $+(\sf{u},\sf{y})\in\Da{A}^\partial$ with $(\sf{u},\sf{x})\sqsubseteq (\sf{u},\Delta(\sf{y}))$.
\item $+(\sf{u},\sf{y})\circ (\sf{u},\sf{x}) = +(\sf{u},\sf{x}\res\Delta(\sf{y}))$ for any $(\sf{u},\sf{x})\in\Da{A}$, $+(\sf{u},\sf{y})\in\Da{A}^\partial$ with $(\sf{u},\sf{x})\sqsubseteq (\sf{u},\Delta(\sf{y}))$.
\item $\circ$ undefined otherwise.
\end{enumerate}
Finally, define a ternary relation $R$ on $\SnuX$ by $R(p,q,r)$ if and only if $p\circ q$ exists and $p\circ q\sqsubseteq r$.
\end{definition}

\begin{theorem}
Let $({\bf S}, {\bf X}, \Upupsilon, \Delta)$ be a dual quadruple. Then $\SnuX$ is the extended Priestley dual of some $\sbp$-algebra.
\end{theorem}

\begin{proof}
By extended Stone-Priestley duality, there exists a Boolean algebra ${\bf B}$ and a $\gmtl$-algebra ${\bf A}$ so that ${\bf S}\cong\Xd{\bf B}$ and ${\bf X}\cong\Xd{\bf A}$, and for simplicity we identify these spaces. Because $\Xdf$ is full and $\Delta$ is a continuous isotone map, there exists a lattice homomorphism $\delta\colon {\bf A}\to {\bf A}$ such that $\Xd{\delta}=\Delta$. We claim that $\delta$ is a wdl-admissible map on ${\bf A}$. To see that $\delta$ is expanding, suppose toward a contradiction that $x\in{\bf A}$ with $x\not\leq\delta(x)$. Then (by the prime ideal theorem for distributive lattices) there exists a prime filter $\fr{x}$ of ${\bf A}$ such that $x\in\fr{x}$ and $\delta(x)\notin\fr{x}$, from which is follows that $x\in\fr{x}$ and $x\notin\delta^{-1}[\fr{x}]=\Delta(\fr{x})$. This contradicts $\Delta$ being expanding, so we must have that $x\leq\delta(x)$ for all $x\in A$. A similar argument shows that $\delta$ is idempotent, and hence a closure operator ($\delta$ is a lattice homomorphism, and thus isotone, by duality). To see that $\delta$ is a nucleus, let $x,y\in{\bf A}$. We must show that $\delta(x)\delta(y)\leq\delta(xy)$. If not, then there exists a prime filter $\fr{z}$ of ${\bf A}$ so that $\delta(x)\delta(y)\in\fr{z}$ and $\delta(xy)\notin\fr{z}$. From this, we have that $\upset\delta(x)\bullet\upset\delta(y)\subseteq\fr{z}$, and by Lemma \ref{lem:filter to prime filter}(2) there exist prime filters $\fr{x}$ and $\fr{y}$ such that $\delta(x)\in\fr{x}$, $\delta(y)\in\fr{y}$ and $\fr{x}\bullet\fr{y}\subseteq\fr{z}$. This gives $R(\fr{x},\fr{y},\fr{z})$, and by hypothesis this implies that $\Delta(\fr{x})\bullet\Delta(\fr{y})\subseteq\Delta(\fr{z})$. But $x\in\delta^{-1}[\fr{x}]=\Delta(\fr{x})$ and $y\in\delta^{-1}[\fr{y}]=\Delta(\fr{y})$, so this gives $xy\in\Delta(\fr{z})$, a contradiction to $\delta(xy)\notin\fr{z}$. It follows that $\delta$ is a wdl-admissible map.

Next, note that by duality we have that for each $u\in{\bf B}$, there exists a homomorphism $\upnu_u\colon {\bf A}\to{\bf A}$ such that $\Xd{\upnu_u}=\upupsilon_{\varphi_{\bf B}(u)}$. For $u\in B$, $x\in A$, define $\join_e$ by
$$u\join_e x = \upnu_u(x).$$
We will show that $\join_e$ is an external join. Toward this goal, note that for all $x\in A$, $u\in{\bf B}$, and $\fr{x}\in\Xd{A}$,
\begin{align*}
\fr{x}\in\upupsilon_{\varphi_{\bf B}(u)}^{-1}[\varphi_{\bf A}(x)] &\iff \upupsilon_{\varphi_{\bf B}(u)}(\fr{x})\in\varphi_{\bf A}(x)\\
&\iff \upnu_u^{-1}[\fr{x}]\in\varphi_{\bf A}(x)\\
&\iff x\in\upnu_u^{-1}[\fr{x}]\\
&\iff \upnu_u(x)\in\fr{x}\\
&\iff \fr{x}\in\varphi_{\bf A}(\upnu_u(x)),
\end{align*}
whence $\upupsilon_{\varphi_{\bf B}(u)}^{-1}[\varphi_{\bf A}(x)]=\varphi_{\bf A}(\upnu_u(x))$. It follows readily from this, together with Definition \ref{def:dualquad}(3), that $\join_e$ satisfies condition (V1), (V2), and (V3). For instance, to see that that each of the maps defined by $\lambda_x(u)=u\join_e x$ (for $x\in A$) gives a lattice homomorphism from ${\bf B}$ to ${\bf A}$ (see (V1)), note that
\begin{align*}
\varphi_{\bf A}(\lambda_x(u\join v)) &= \varphi_{\bf A}(\upnu_{u\join v}(x))\\
&=\upupsilon_{\varphi_{\bf B}(u\join v)}^{-1}[\varphi_{\bf A}(x)]\\
&=\upupsilon_{\varphi_{\bf B}(u)\cup\varphi_{\bf B}(v)}^{-1}[\varphi_{\bf A}(x)]\\
&=\upupsilon_{\varphi_{\bf B}(u)}^{-1}[\varphi_{\bf A}(x)]\cup\upupsilon_{\varphi_{\bf B}(v)}^{-1}[\varphi_{\bf A}(x)]\\
&=\varphi_{\bf A}(\upnu_{u}(x))\cup \varphi_{\bf A}(\upnu_{v}(x))\\
&= \varphi_{\bf A}(\lambda_x(u))\cup \varphi_{\bf A}(\lambda_x(v)),
\end{align*}
and hence $\lambda_x(u\join v)=\lambda_x(u)\join\lambda_x(v)$ for any $x\in A$, $u,v\in B$. Similar reasoning using the fact that $({\sf U},{\sf X})\mapsto\upupsilon^{-1}_{\sf U}[{\sf X}]$ is an algebraic quadruple shows that $\join_e$ satisfies (V1), (V2), and (V3). It follows that $({\bf B},{\bf A},\join_e,\delta)$ is an algebraic quadruple. Thus $\SnuX$ is the extended Priestley space of ${\bf B}\otimes_e^\delta {\bf A}$ by construction.
\end{proof}

\begin{lemma}\label{lem:splits up}
Let ${\bf A}$ be an $\sbp$-algebra, and let $\fr{x}\in\Xd{\rad{\bf A}}$. As usual, for each $u\in\bool{\bf A}$ set
$$\mu_u(\fr{x})=\{x\in\rad{\bf A} : u\join x\in\fr{x}\}.$$
Then for each $u,v\in\bool{\bf A}$ we have each of the following.
\begin{enumerate}
\item $\mu_{u\join v}(\fr{x})$ is one of $\mu_u(\fr{x})$ or $\mu_v(\fr{x})$.
\item $\mu_{u\meet v}(\fr{x})$ is one of $\mu_u(\fr{x})$ or $\mu_v(\fr{x})$.
\item $\mu_u(\fr{x})=\fr{x}$ or $\mu_{\neg u}(\fr{x})=\fr{x}$.
\item $\mu_u (\fr{x})=\fr{x}$ or $\mu_u(\fr{x})=\rad{\bf A}$.
\item $\mu_u(\mu_u(\fr{x}))=\mu_u(\fr{x})$.
\end{enumerate}
\end{lemma}

\begin{proof}
Note that for any $u,v\in\bool{\bf A}$ and $x\in\rad{\bf A}$, we have that both of $x\join u\in\fr{x}$ and $x\join v\in\fr{x}$ imply $x\join u\join v\in\fr{x}$ since $\upset\fr{x}=\fr{x}$. Likewise, $x\join (u\meet v)\in\fr{x}$ implies that $x\join u\in\fr{x}$ and $x\join v\in\fr{x}$. It follows from these observations that $\mu_u(\fr{x}),\mu_v(\fr{x})\subseteq\mu_{u\join v}(\fr{x})$ and $\mu_{u\meet v}(\fr{x})\subseteq\mu_u(\fr{x}),\mu_v(\fr{x})$.
 
To prove (1), toward a contradiction, suppose that each of the inclusions $\mu_u(\fr{x}),\mu_v(\fr{x})\subseteq\mu_{u\join v}(\fr{x})$ is strict. Then there exist $x,y\in\rad{\bf A}$ such that $x\join u\join v,y\join u\join v\in\fr{x}$, but $x\join u\notin\fr{x}$ and $x\join v\notin\fr{x}$. Because $\fr{x}$ is upward-closed and $x\join u\join v,y\join u\join v\in\fr{x}$, we obtain that $x\join y\join u\join v\in\fr{x}$. But $\fr{x}$ is prime in $\rad{\bf A}$ and $x\join u, y\join v\in\rad{\bf A}$, so $(x\join u)\join (y\join v) = x\join y\join u\join v\in\fr{x}$ implies $x\join u\in\fr{x}$ or $y\join v\in\fr{x}$. This contradicts the hypothesis, so either $\mu_{u\join v}(\fr{x})=\mu_u(\fr{x})$ or $\mu_{u\join v}(\fr{x})=\mu_v(\fr{x})$, proving (1).

To prove (2), again suppose that both of the inclusions $\mu_{u\meet v}(\fr{x})\subseteq\mu_u(\fr{x}),\mu_v(\fr{x})$ are strict. Then there exist $x,y\in\rad{\bf A}$ such that $x\join (u\meet v),y\join (u\meet v)\notin\fr{x}$ but $x\join u\in\fr{x}$ and $y\join v\in\fr{x}$. By distributivity, we have $(x\join u)\meet (x\join v)\notin\fr{x}$, whence since $x\join u\in\fr{x}$ we have $x\join v\notin\fr{x}$. Similarly, from $(y\join u)\meet (y\join v)\notin\fr{x}$ and $y\join v\in\fr{x}$ we get $y\join u\notin\fr{x}$. Because $\fr{x}$ is prime, $x\join v,y\join u\notin\fr{x}$ implies $x\join y\join y\join v\notin\fr{x}$. But this contradicts $x\join u\in\fr{x}$ since $x\join u\leq x\join y\join u\join v$ and $\fr{x}$ is upward-closed, which proves (2).

For (3), note that $\fr{x}=\mu_0(\fr{x})=\mu_{u\meet\neg u}(\fr{x})$, which by (2) is one of $\mu_u(\fr{x})$ or $\mu_{\neg u}(\fr{x})$.

For (4), suppose that $\mu_u(\fr{x})\neq\fr{x}$ and $\fr{x}\neq\rad{\bf A}$. By (3) this implies that $\mu_{\neg u}(\fr{x})=\fr{x}$. Also, $\rad{\bf A}=\mu_1(\fr{x})=\mu_{u\join\neg u}(\fr{x})$, and by (1) one the latter is one of $\mu_u(\fr{x})$ or $\mu_{\neg u}(\fr{x})$. The latter is excluded by the assumption that $\fr{x}\neq\rad{\bf A}$, whence $\rad{\bf A}=\mu_u(\fr{x})$.

Finally, (5) is immediate from (4).
\end{proof}.

\begin{theorem}
Let ${\bf Y}$ be the extended Priestley dual of an $\sbp$-algebra. Then there exists a dual quadruple $({\bf S}, {\bf X}, \Upupsilon, \Delta)$ such that ${\bf Y}\cong\SnuX$.
\end{theorem}

\begin{proof}
Let ${\bf A}=(A,\meet,\join,\cdot,\to,1,0)$ be an $\sbp$-algebra such that ${\bf Y}=\Xd{\bf A}$. Set ${\bf S}:=\Xd{\bool{\bf A}}$ and ${\bf X}:=\Xd{\rad{\bf A}}$. Further define maps $\Delta\colon{\bf X}\to{\bf X}$ by $\Delta(\fr{x})=\{x\in\rad{\bf A} : \neg\neg x\in\fr{x}\}$ and, for each ${\sf U}\in\Ad{\bf S}$, $\upupsilon_{\sf U}\colon{\bf X}\to {\bf X}$ by $\upupsilon_{\sf U}(\fr{x})=\mu_{\varphi^{-1}({\sf U})}(\fr{x})=\{x\in\rad{\bf A} : \varphi^{-1}({\sf U})\join x\in\fr{x}\}$ (where $\varphi$ is the usual isomorphism from $\bool{\bf A}$ to $\Ad{\bf S}$). Set $\Upupsilon = \{\upupsilon_{\sf U}\}_{{\sf U}\in\Ad{\bf S}}$. We will show that $({\bf S},{\bf X},\Upupsilon,\Delta)$ is a dual quadruple.

The requirements (1) and (2) of Definition \ref{def:dualquad} are met by hypothesis. To see that (4) is satisfied, let $\fr{x},\fr{y},\fr{z}\in{\bf X}$ such that $R(\fr{x},\fr{y},\fr{z})$. This gives that $\fr{x}\bullet\fr{y}\subseteq\fr{z}$. We will show that $\Delta(\fr{x})\bullet\Delta(\fr{y})\subseteq\Delta(\fr{z})$, so let $z\in\Delta(\fr{x})\bullet\Delta(\fr{y})$. Then there exists $x\in\Delta(\fr{x})$ and $y\in\Delta(\fr{y})$ such that $x\cdot y\leq z$. From this we have that $\neg\neg x\in\fr{x}$ and $\neg\neg y\in\fr{y}$, so $\neg\neg x\cdot\neg\neg y\in \fr{x}\bullet\fr{y}\subseteq\fr{z}$. Because $\neg\neg x\cdot\neg\neg y\leq \neg\neg (x\cdot y)$, this provides that $\neg\neg(x\cdot y)\in\fr{z}$, whence $\neg\neg z\in\fr{z}$. It follows that $z\in\Delta(\fr{z})$ as desired, concluding the proof of (4).

To show that requirement (3) of Definition \ref{def:dualquad} is met, note that for each ${\sf U}\in\Ad{\bf X}$ we have that the map $\upupsilon_{\sf U}$ is a morphism of $\gmtl^\tau$ because $\upupsilon_{\sf U}$ is the dual of the $\gmtl$-morphism $x\mapsto\varphi^{-1}({\sf U})\join x$. For each ${\sf U}\in\Ad{\bf S}$ and ${\sf X}\in\Ad{\bf X}$, define
$$\join_e ({\sf U},{\sf X}) = \upupsilon_{\sf U}^{-1}[{\sf X}].$$
We will show that $\join_e\colon\Ad{\bf S}\times\Ad{\bf X}\to\Ad{\bf X}$ defines an external join, i.e., that it satisfies (V1), (V2), and (V3) of Definition \ref{def:algquad}. Note that for fixed ${\sf U}\in\Ad{\bf S}$, the map $\join_e({\sf U},-)$ is an endomorphism of $\Ad{\bf X}$ by extended Priestley duality. Let ${\sf X}\in\Ad{\bf X}$, and set $\lambda_{\sf X}({\sf U})=\join_e({\sf U},{\sf X})$. Let ${\sf U},{\sf V}\in\Ad{\bf S}$. Note that $\lambda_{\sf X}({\sf U})\cup\lambda_{\sf X}({\sf V})\subseteq\lambda_{\sf X}({\sf U}\cup{\sf V})$ follows easily from the fact that $\varphi^{-1}$ is a lattice homomorphism together with the fact that ${\sf X}$ is upward-closed, whereas the reverse inclusion may be obtained Lemma \ref{lem:splits up}(1). Thus $\lambda_{\sf X}({\sf U})\cup\lambda_{\sf X}({\sf V}) = \lambda_{\sf X}({\sf U}\cup{\sf V})$. That $\lambda_{\sf X}({\sf U}\cap {\sf V})=\lambda_{\sf X}({\sf U})\cap\lambda_{\sf X}({\sf V})$ follows similarly, and this finishes the proof of (V1).

To prove (V2), note that $\upupsilon^{-1}_\emptyset$ being the identity on $\Ad{\bf X}$ follows immediately from the fact that $\mu_0(\fr{x})=\fr{x}$ for any $\fr{x}\in {\bf X}$. Likewise, $\upupsilon^{-1}_S({\sf X})=X$ for any ${\sf X}\in\Ad{\bf X}$ follows from the fact that $\mu_1(\fr{x})=\rad{\bf A}$ for any $\fr{x}\in{\bf X}$.

For (V3), we must show that
$$\upupsilon^{-1}_{\sf U}[{\sf X}]\cup\upupsilon^{-1}_{\sf V}[{\sf Y}]=\upupsilon^{-1}_{{\sf U}\cup{\sf V}}[{\sf X}\cup{\sf Y}]=\upupsilon^{-1}_{\sf U}[\upupsilon^{-1}_{\sf V}[{\sf X}\cup{\sf Y}]].$$
Observe that an easy argument shows
$$\mu_{\varphi^{-1}({\sf U})\cup\varphi^{-1}({\sf V})}(\fr{x})=\mu_{\varphi^{-1}({\sf V})}(\mu_{\varphi^{-1}({\sf U})}(\fr{x})),$$
from whence it follows that $\upupsilon^{-1}_{{\sf U}\cup{\sf V}}[{\sf X}\cup{\sf Y}]=\upupsilon^{-1}_{\sf U}[\upupsilon^{-1}_{\sf V}[{\sf X}\cup{\sf Y}]]$.

Next let $\fr{x}\in\upupsilon^{-1}_{\sf U}[{\sf X}]\cup\upupsilon^{-1}_{\sf V}[{\sf Y}]$. Then $\fr{x}\in \upupsilon^{-1}_{\sf U}[{\sf X}]$ or $\fr{x}\in\upupsilon^{-1}_{\sf V}[{\sf Y}]$, i.e., $\mu_{\varphi^{-1}({\sf U})}(\fr{x})\in {\sf X}$ or $\mu_{\varphi^{-1}({\sf V})}(\fr{x})\in{\sf Y}$. Since ${\sf X},{\sf Y}$ are upward closed, this implies that $\mu_{\varphi^{-1}({\sf U}\cup{\sf V})}(\fr{x})\in {\sf X}, {\sf Y}$, so certainly $\upupsilon_{{\sf U}\cup{\sf V}}(\fr{x})\in {\sf X}\cup {\sf Y}$, thus
$\fr{x}\in\upupsilon^{-1}_{{\sf U}\cup{\sf V}}[{\sf X}\cup{\sf Y}]$, and $\upupsilon_{\sf U}[{\sf X}]\cup\upupsilon_{\sf V}[{\sf Y}]\subseteq\upupsilon_{{\sf U}\cup{\sf V}}[{\sf X}\cup{\sf Y}]$.

For the final inclusion, let $\fr{x}\in \upupsilon^{-1}_{{\sf U}\cup{\sf V}}[{\sf X}\cup{\sf Y}]=\upupsilon^{-1}_{\sf U}[\upupsilon^{-1}_{\sf V}[{\sf X}\cup{\sf Y}]]$. Then $\mu_{\varphi^{-1}({\sf U})\cup\varphi^{-1}({\sf V})}(\fr{x})\in {\sf X}\cup{\sf Y}$. By Lemma \ref{lem:splits up}(1), we assume without loss of generality that
$$\mu_{\varphi^{-1}({\sf U})\cup\varphi^{-1}({\sf V})}(\fr{x})=\mu_{\varphi^{-1}({\sf U})}(\fr{x}).$$
Then $\mu_{\varphi^{-1}({\sf U})}(\fr{x})\in {\sf X}\cup {\sf Y}$, whence $\mu_{\varphi^{-1}({\sf U})}(\fr{x})\in {\sf X}$ or $\mu_{\varphi^{-1}({\sf U})}(\fr{x})\in {\sf Y}$. In the former case $\fr{x}\in \upupsilon^{-1}_{\sf U}[{\sf X}]\cup\upupsilon^{-1}_{\sf V}[{\sf Y}]$ is immediate, so assume that $\mu_{\varphi^{-1}({\sf U})}(\fr{x})\notin {\sf X}$. Because $\rad{\bf A}\in {\sf X}$, it follows that $\mu_{\varphi^{-1}({\sf U})}(\fr{x})\neq \rad{\bf A}$ and hence from Lemma \ref{lem:splits up}(4) that $\mu_{\varphi^{-1}({\sf U})}(\fr{x})=\fr{x}\in {\sf Y}$. Because ${\sf Y}$ is upward-closed and $\fr{x}\subseteq\mu_{\varphi^{-1}({\sf V})}(\fr{x})$, this gives that $\mu_{\varphi^{-1}({\sf V})}(\fr{x})\in {\sf Y}$, and thus $\fr{x}\in\upupsilon^{-1}_{\sf U}[{\sf X}]\cup\upupsilon^{-1}_{\sf V}[{\sf Y}]$ in any case. This yields (V3), and therefore that $\SnuX$ is a dual quadruple.

To complete the proof, note that $\SnuX\cong \Xd{\bf A}\cong {\bf Y}$ via the isomorphism $\alpha$ of Section \ref{sec:filterpairs} and the construction of $\SnuX$.
\end{proof}

\section{Conclusion}

Constructions along the lines of \cite{AguzFlamUgol} yield significant insight into the structure of the various classes of algebras to which they apply, and commensurately into the logics for which these algebras supply semantics. Duality-theoretic treatments of these constructions in turn provide pictorial insight, suggesting new and intuitive ways of understanding constructions and deepening our overall understanding of the structures in play. The foregoing dualized construction reveals the governing role of the family of maps $\{\upnu_u\}_{u\in\bool{\bf A}}$ in constructing an $\sbp$-algebra ${\bf A}$ from its corresponding algebraic quadruple, and particularly the role of this family of maps in determining the lattice-reduct of ${\bf A}$ (which is not obvious merely from the definition of the lattice operations as rendered in Section \ref{sec:quadruples}). Duality-theoretic tools for residuated lattices, along the lines of those articulated in Section \ref{sec:duality}, play an important part in making analyses of the foregoing kind feasible. Additions to the duality-theoretic toolkit for residuated lattices promise to further deepen our understanding of their algebraic structure, and hence also of substructural logics generally.



\bibliographystyle{spmpsci}
\bibliography{bibDualTriples}

\end{document}